\documentclass[12pt]{article}
\usepackage{amsmath,amssymb}
\usepackage{amscd}
\usepackage{comment}
\usepackage{color}
\usepackage{enumerate}
\usepackage{cases}
\usepackage{mathtools}
\usepackage[all]{xy}
\usepackage{amsthm}
\usepackage{pb-diagram}
\usepackage[top=25truemm,bottom=25truemm,left=20truemm,right=20truemm]{geometry}
\theoremstyle{definition}
\newtheorem{defi}{Definition}
\newtheorem{theo}[defi]{Theorem}
\newtheorem{lemm}[defi]{Lemma}
\newtheorem{prop}[defi]{Proposition}

\newtheorem{rem}[defi]{Remark}
\newtheorem{que}[defi]{Question}

\def\det{\mathrm{det}}

\def\SO{\mathrm{SO}}

\def\SL{\mathrm{SL}}

\def\SU{\mathrm{SU}}

\def\End{\mathrm{End}}

\def\Vol{\mathrm{Vol}}

\def\refe{\ast}

\def\vol{\mathrm{vol}}

\def\loc{\mathrm{loc}}

\def\R{{\mathbb R}}

\def\C{{\mathbb C}}
\def\N{{\mathbb N}}

\def\K{{\mathbb K}}

\def\D{{\mathbb D}}

\def\inum{{\sqrt{-1}}}

\providecommand{\keywords}[1]
{
\small 
\textbf{{Keywords---}} #1
}
\providecommand{\MSC}[1]
{
\small 
\textbf{{MSC---}} #1
}

\begin{document}

\title {Shannon entropy for harmonic metrics on cyclic Higgs bundles}
\author {Natsuo Miyatake}
\date{}
\maketitle
\begin{abstract} 
Let $X$ be a Riemann surface, $K_X \rightarrow X$ the canonical bundle, and $T_X\coloneqq K_X^{-1}\rightarrow X$ the dual bundle of the canonical bundle. For each integer $r \geq 2$, each $q \in H^0(K_X^r)$, and each choice of the square root $K_X^{1/2}$ of the canonical bundle, we canonically obtain a Higgs bundle, which is called a cyclic Higgs bundle. A diagonal harmonic metric $h = (h_1, \dots, h_r)$ on a cyclic Higgs bundle yields $r-1$ Hermitian metrics $H_1, \dots, H_{r-1}$ on $T_X\rightarrow X$, defined as $H_j \coloneqq h_j^{-1} \otimes h_{j+1}$ for each $j=1,\dots, r-1$, while $h_1$, $h_r$, and $q$ yield a degenerate Hermitian metric $H_r$ on $T_X \rightarrow X$. 
The $r$-differential $q$ induces a subharmonic weight function $\phi_q=\frac{1}{r}\log|q|^2$ on $K_X\rightarrow X$, and the diagonal harmonic metric depends solely on this weight function $\phi_q$. In the previous papers, the author introduced and studied the extension of harmonic metrics associated with an arbitrary subharmonic weight function $\varphi$, which also constructs $r-1$ Hermitian metrics $H_1,\dots, H_{r-1}$ and a degenerate Hermitian metric $H_r$ on $T_X\rightarrow X$. In this paper, for each non-zero real parameter $\beta$, we introduce a function, which we call entropy, that quantifies the degree of mutual misalignment of the Hermitian metrics $H_1,\dots, H_r$. 
By extending the estimates established by Dai–-Li and Li–-Mochizuki to general subharmonic weight functions, we obtain upper and lower bounds on the entropy for $\beta>0$, assuming that $H_1,\ldots,H_{r-1}$ are complete, and that the weight function is neither induced by a complete flat K\"ahler metric nor identically equal to $-\infty$. The upper bound is $\log r$, the constant entropy obtained when the weight function is induced by a complete flat K\"ahler metric, in which case $H_1=\cdots=H_r$. The lower bound is the constant entropy associated with the weight function $\varphi\equiv-\infty$. Furthermore, for arbitrary nonzero real $\beta$, we determine the asymptotic behavior of the difference between these two entropy values as $r\to\infty$, showing that it converges to a finite real number if and only if $\beta>-1$.
\end{abstract}

\MSC{30C15, 31A05, 53C07}


\section{Introduction and main theorem}
As is well known, every compact connected Riemann surface possesses a constant Gaussian curvature K\"ahler metric. Among them, constant negative Gaussian curvature K\"ahler metrics and flat Gaussian curvature K\"ahler metrics are related to Hitchin's seminal example of Higgs bundles \cite{Hit1}. Hitchin's example is as follows: Let $X$ be a compact connected Riemann surface and $K_X\rightarrow X$ the canonical bundle. For each choice of the square root $K_X^{1/2}$ of the canonical bundle and each $q\in H^0(K_X^2)$, we can associate the following Higgs bundle $(\K_2,\Phi(q))\rightarrow X$:
\begin{align*}
&\K_2\coloneqq K_X^{1/2}\oplus K_X^{-1/2}, \\ 
&\Phi(q)\coloneqq 
\left(
\begin{array}{cc}
0 & q \\ 
1 &0 
\end{array}
\right)\in H^0(\End \K_2\otimes K_X).
\end{align*}
This example was introduced by Hitchin \cite[Section 11]{Hit1}, especially on a compact hyperbolic Riemann surface. Suppose that $X$ is a compact hyperbolic Riemann surface. Then $(\K_2,\Phi(q))\rightarrow X$ is always stable (cf. \cite{Hit1}), and thus there uniquely exists a solution $h$ to the Hitchin equation satisfying $\det(h)=1$. Moreover, from the symmetry of the Higgs field and the uniqueness of the solution to the Hitchin equation, the metric $h$ splits as $h=(h_1,h_1^{-1})$. Let $T_X\rightarrow X$ be the dual bundle of $K_X\rightarrow X$. The solution $h=(h_1,h_1^{-1})$ induces a metric $H_1\coloneqq h_1^{-2}$ on $T_X\rightarrow X$ and one degenerate metric (cf. \cite[Theorem (11.2)]{Hit1}) $H_2\coloneqq h_1^2\otimes h_q$ on $T_X\rightarrow X$, where $h_q$ denotes the degenerate Hermitian metric on $K_X^{-2}\rightarrow X$ induced by $q$, defined as follows:
\begin{align}
(h_q)_x(u,v)\coloneqq \langle q_x,u\rangle\overline{\langle q_x,v\rangle} \ \text{for $x\in X$, $u,v\in (K_X^{-2})_x$}. \label{uv}
\end{align}
If $q$ is identically zero, then the K\"ahler metric induced by $H_1$ is a constant negative Gaussian curvature metric (cf. \cite[Section 11]{Hit1}). If $X$ is an elliptic curve and $q$ is a non-zero section, then we can easily find a solution to the Hitchin equation; $h=(h_1, h_1^{-1})=(h_q^{-1/4}, h_q^{1/4})$ solves the Hitchin equation, where $h_q$ is a Hermitian metric on $K_X^{-2}\rightarrow X$ defined as in (\ref{uv}). In this case, if we define $H_1=h_1^{-2}$ and $H_2=h_1^2\otimes h_q$ in the same way as in the case where $X$ is hyperbolic, we obtain $H_1=H_2=h_q^{1/2}$, and the K\"ahler metrics induced by them are flat Gaussian curvature metrics. From the above, we can observe that the solution of the Hitchin equation for Hitchin's example contains both constant negative Gaussian curvature K\"ahler metrics and flat Gaussian curvature K\"ahler metrics. They are essentially equivalent to the harmonic metrics on Hitchin's example for the two extremal cases; the case where $q$ is identically zero, and the case where it has no zeros at all. We can further observe that in the case where $H_1$ yields a hyperbolic metric, $H_1$ and $H_2$ do not coincide anywhere on $X$ since $H_2$ is identically zero, while in the case where $H_1$ is flat, $H_1$ and $H_2$ coincide everywhere. Hitchin's example can naturally be generalized to harmonic metrics on higher rank cyclic Higgs bundles (cf. \cite{Bar1, Hit2}). Similar observations as above can be made for harmonic metrics on the general cyclic Higgs bundles as will be explained in Section \ref{2}.

In previous papers \cite{Miy2, Miy3}, the author proposed a new direction of research for cyclic Higgs bundles that considers harmonic metrics even on cyclic ramified coverings and Hermitian metrics that solve the extended Hitchin equation with a coefficient obtained as the limit of infinitely increasing covering order. In other words, this means that we consider a more general semipositive singular Hermitian metric $e^{-\varphi}h_\refe$ on the canonical bundle instead of the singular metric induced by $q$. The intention was to connect the potential theory, especially the theory of zero-point configurations of holomorphic functions and holomorphic sections which has been the subject of much research for a very long time (cf. \cite{BCHM1, FKMR1, Ran1, ST1, SZ1}), with the theory of harmonic metrics on cyclic Higgs bundles. In \cite{Miy2, Miy3}, the author introduced an elliptic equation that extends the Hitchin equation for the ordinary cyclic Higgs bundle to be adjunct to an arbitrary semipositive singular metric $e^{-\varphi} h_\refe$ on $K_X\rightarrow X$, and discussed some fundamental properties of the equation. As with the usual cyclic Higgs bundles, by taking the adjacent components of the solution $(h_1,\dots, h_r)$ to the elliptic equation, we obtain $r-1$ Hermitian metrics $H_1,\dots, H_{r-1}$ on $T_X\rightarrow X$ and $h_1, h_r$ and $e^{-\varphi}h_\refe$ yield one degenerate Hermitian metric $H_r$ on $T_X\rightarrow X$ (see Section \ref{2.2}). 

In this paper, we introduce a new function which we call {\it entropy}. It is constructed from a non-zero real number $\beta$ and the metrics $H_1,\dots,H_r$. Throughout this paper, we suppose that $H_1,\dots,H_{r-1}$ are all complete. Our first main theorem concerns the uniform estimate of the entropy from above and below, assuming that the weight function is neither induced by a complete flat K\"ahler metric nor identically equal to $-\infty$:

\begin{theo}\label{main theorem 1}
{\it
Let $e^{-\varphi}h_\refe$ be a semipositive singular Hermitian metric on $K_X\rightarrow X$ that is not identically $\infty$ and is not induced by a complete flat K\"ahler metric.
Then, for any positive real number $\beta$, the entropy $S(r,\beta,\varphi)$ satisfies the following uniform bounds:
\begin{align}
S_{r,\beta}\leq S(r,\beta,\varphi)(x)<\log r
\quad\text{for every $x\in X$}, \label{SrbSrp}
\end{align}
where $S_{r,\beta}$ is the entropy associated with the weight function that is identically $-\infty$. Moreover, equality in the lower bound for $S(r,\beta,\varphi)(x)$ holds if and only if $r=2,3$ and $\varphi(x)=-\infty$.
}
\end{theo}
 The upper bound is $\log r$, the constant entropy obtained when the weight function is induced by a complete flat K\"ahler metric, in which case $H_1=\cdots=H_r$. The lower bound is the constant entropy associated with the weight function $\varphi\equiv-\infty$ (see Section \ref{4.2} for more details). Theorem \ref{main theorem 1} asserts that, under assumptions excluding these trivial solutions, the entropy is bounded below and above by the entropy values associated with the hyperbolic and Euclidean cases, respectively. In our second main theorem, for arbitrary nonzero real $\beta$, we determine the asymptotic behavior of the difference between these two entropy values as $r\to\infty$: 
\begin{theo}\label{main theorem 2} {\it The following asymptotic formula holds:
\begin{align*}
\lim_{r\to\infty} (S_{r,\beta}-\log r)=
\begin{cases}
-\frac{2\beta d_\beta}{c_\beta}+\log(c_\beta) \ &\text{if $\beta>-1$} \\
-\infty \ &\text{if $\beta\leq -1$},
\end{cases}
\end{align*}
where $c_\beta$ and $d_\beta$ are defined as follows:
\begin{align*}
&c_\beta\coloneqq \int_0^1s^\beta(1-s)^\beta ds, \\
&d_\beta\coloneqq \int_0^1s^\beta(1-s)^\beta \log s\ ds.
\end{align*}
}
\end{theo}
For $\beta<0$, the bounds $0\le S(r,\beta,\varphi)(x)<\log r$ hold under the same assumptions on the weight function as in Theorem \ref{main theorem 1}. Alternatively, if we consider the ``conditional" entropy 
\begin{align*}
&\widetilde{S}(r, \beta, \varphi) \coloneqq -\sum_{j=1}^{r-1}\widetilde{p}_j(r,\beta,\varphi)\log \widetilde{p}_j(r,\beta,\varphi), \\
& \widetilde{p}_j(r,\beta,\varphi)\coloneqq \frac{p_j(r,\beta,\varphi)}{1-p_0(r,\beta,\varphi)}\ \text{for $j=1,\dots, r-1$}, 
\end{align*} 
then $S_{r, \beta}< \widetilde{S}(r,\beta,\varphi)<\log(r-1)$ holds for $r\geq 4$. 

 The proof of Theorem \ref{main theorem 1} proceeds by extending to general subharmonic weight functions the estimates of Dai--Li and Li--Mochizuki \cite{DL1, DL2, LM1} for the ratios between $H_j$ and $H_{j+1}$, $j=1,\dots,r-1$, originally established in the case where $H_1,\dots,H_r$ arise from harmonic metrics, and then combining these estimates with the Schur concavity of the Shannon entropy. Analogous bounds can also be obtained for other Schur-concave functions.

For $\beta<-1$, as the proof shows, $S_{r,\beta}$ itself converges to a finite limit $\beta(\zeta^\prime(-\beta)/\zeta(-\beta))+\log(2\zeta(-\beta))$ as $r\to\infty$. This convergence and further refinements of Theorem 2, including correction terms, will be discussed in \cite{Miy8}.

The most promising direction for the further development of this entropy function, as envisioned by the author, lies in its application to the {\it approximation theory of equilibrium metrics and measures}. This theory has been studied in various fields, including complex analysis, probability theory, and dynamical systems (see, e.g., \cite{BCHM1, Ber1, Ber2, BB1, BBN1, BL1, BL2, BL3, LNV1, Ran1, ST1}). This aspect is not discussed in detail in the present paper; a full account is given in \cite{Miy7}.

The parameters $r$ and $\beta$ may be regarded as specifying the ``method of observation or measurement'' applied to $\varphi$. As mentioned above, one possible direction of research is to study how the entropy behaves as $\varphi$ varies while $r$ and $\beta$ are held fixed. Another direction is to fix $\varphi$ and investigate how the entropy changes as $r$ and $\beta$, which encode the ``method of observation or measurement,'' vary. As one further development of Theorem \ref{main theorem 2}, it is particularly natural to ask the following question:

\begin{que}
For each point $x\in X$, what is the asymptotic behavior of $S(r,\beta,\varphi)(x)$ as $r\to\infty$? 
\end{que}
\noindent
Regarding ordinary harmonic metrics, the author is not aware of any research aimed at investigating the asymptotic behavior of the harmonic metric in some sense in the limit of the parameter $r$, which controls the size of the symmetric space to which the harmonic map goes, becoming infinitely large. For research on the $r$-dependence of various quantities constructed from harmonic metrics, see, for example, \cite{DL1, DL2, Li0, LM3, Ma1}. In particular, by the estimate of Dai--Li \cite[Theorem 1.5]{DL2}, the sectional curvature of the image of the harmonic map associated with a cyclic Higgs bundle tends to vanish uniformly for all $q \in H^0(K_X^r)$ as the rank $r$ goes to infinity.

As the final part of the introduction, we make a remark on the difference between the entropy and its upper bound. Consider the case where the weight function $\varphi$ equals $\phi_q=\frac{1}{r}\log|q|^2_{h_\refe}$ for some $q\in H^0(K_X^r)$. Suppose that $X$ is a hyperbolic Riemann surface and $\sup_X|q|_{g_X}$ is finite for the complete hyperbolic K\"ahler metric $g_X$ (for equivalent conditions, see \cite[Theorem 4.8]{DL3}). 
For simplicity, we set $\beta = 1$.
In this case, by applying Dai--Li's result \cite[Theorem 4.8]{DL3} (see also \cite{BH1, LT1, Wan1} for works on the cases of smaller rank), entropy can be uniformly bounded from above by $\log r-\delta$, where $\delta$ is a positive constant which depends only on $\sup_X|q|_{g_X}$. The author would like to thank Qiongling Li for pointing this out when he sent her the first draft of the present paper. While the above improvement to the upper bound of the entropy will not be discussed in depth in this paper, it will be addressed in a subsequent work \cite{Miy4}.

The paper is organized as follows. In Section \ref{2}, we explain some background. In Section \ref{Shannon}, we review some basic properties of Shannon entropy and Schur-concave functions. In Section \ref{CYMMVI}, we review the Cheng--Yau maximum principle and the mean value inequality. In Section \ref{4}, we define entropy. In Section \ref{6}, we prove our main theorems.

\medskip
\noindent
{\bf Acknowledgements.} 

\noindent
I am grateful to Ryoichi Kobayashi for questioning the direction of moving the parameter $r$, which controls the size of the target symmetric space of harmonic maps. 
I would like to express my deep gratitude to Qiongling Li for answering my many questions, for reading the first draft of this paper, and for pointing out an improvement in the upper bound of the entropy when the $r$-differential $q$ is bounded, with respect to the complete hyperbolic K\"ahler metric, by a constant depending only on the norm of $q$ measured by the hyperbolic metric. 
I would also like to extend my sincere thanks to Toshiaki Yachimura for his advice on the readability of this paper, for our many discussions, and for his invaluable help, especially in the definition of entropy. In preparing version 3 of this paper, I used conversations with ChatGPT (OpenAI) for help with revising the manuscript. In particular, ChatGPT helped me identify a gap in the proof of Theorem \ref{main theorem 2} in version 2. This work was supported by the Grant-in-Aid for Early Career Scientists (Grant Number: 24K16912).

\section{Backgrounds}\label{2}
\subsection{Real, diagonal, and complete Hermitian metrics}\label{rdc}
Let $X$ be a connected, possibly non-compact Riemann surface and $K_X\rightarrow X$ the canonical bundle. We choose a square root $K_X^{1/2}$ of the canonical bundle $K_X\rightarrow X$. Let $\K_r$ be a holomorphic vector bundle of rank $r$ defined as follows:
\begin{align*}
\K_r\coloneqq \bigoplus_{j=1}^r K_X^{\frac{r-(2j-1)}{2}} = K_X^{\frac{r-1}{2}} \oplus K_X^{\frac{r-3}{2}} \oplus \cdots \oplus K_X^{-\frac{r-3}{2}} \oplus K_X^{-\frac{r-1}{2}}.
\end{align*}
A Hermitian metric $h$ on $\K_r\rightarrow X$ is said to be {\it real} (cf. \cite{Hit2}) if the following bundle isomorphism $S$ is isometric with respect to $h$:
\begin{align*}
S\coloneqq \left(
\begin{array}{ccc}
& & 1 \\ 
& \reflectbox{$\ddots$} & \\ 
1 & &
\end{array}
\right): \K_r\rightarrow \K_r^{\vee},
\end{align*}
where $\K_r^{\vee}$ denotes the dual bundle of $\K_r$. Also, $h$ is called {\it diagonal} (cf. \cite{Bar1, Col1, DL2}) if it splits as $h=(h_1,\dots, h_r)$. Let $T_X\rightarrow X$ be the dual bundle of $K_X\rightarrow X$. By taking the difference between the adjacent components in a diagonal Hermitian metric $h=(h_1,\dots, h_r)$, we obtain $r-1$ Hermitian metrics $H_1,\dots, H_{r-1}$ on $T_X\rightarrow X$:
\begin{align*}
H_j\coloneqq h_j^{-1}\otimes h_{j+1}.
\end{align*}
A diagonal Hermitian metric $h$ is real if and only if $h_j=h_{r-j+1}^{-1}$ for all $j=1,\dots, r$. This is also equivalent to the metrics $H_1,\dots, H_{r-1}$ satisfying $H_j=H_{r-j}$ for all $j=1,\dots, r-1$. We say that a diagonal Hermitian metric $h$ is {\it complete} (cf. \cite{LM1}) if the K\"ahler metrics induced by $H_1,\dots, H_{r-1}$ are all complete. 
\begin{rem} Note that the notion of completeness introduced above is considerably broader than the original notion in \cite{LM1}. In \cite{LM1}, a reference K\"ahler metric is fixed, thereby allowing diagonal Hermitian metrics on $\K_r$ to be identified with $r$-tuples of functions. The term ``complete solution'' is then used for a diagonal Hermitian metric $h$ on $\K_r$, viewed as an $r$-tuple of functions, that satisfies $\det(h)=1$ and solves the Hitchin equation for a cyclic Higgs bundle. In contrast, we use the term ``complete'' for arbitrary diagonal Hermitian metrics, without requiring them to be harmonic or imposing any condition on their determinants, although the main theorem and its proof deal exclusively with metrics satisfying $\det(h)=1$.
\end{rem} 
\begin{rem}
In this paper, we allow K\"ahler metrics on Riemann surfaces and Hermitian metrics on $\K_r$ to be continuous but not necessarily smooth. 
Whenever we say that a (possibly non-smooth) K\"ahler metric is complete, we mean that the induced distance function defines a complete metric space.
\end{rem}
\subsection{Harmonic metrics on cyclic Higgs bundles}\label{2.1}
Let $q$ be a holomorphic section of $K_X^r\rightarrow X$. Then we associate a holomorphic section $\Phi(q)\in H^0(\End \K_r\otimes K_X)$, which is called a Higgs field, as follows:
\begin{align*}
\Phi(q)\coloneqq 
\left(
\begin{array}{cccc}
0 & && q \\ 
1 & \ddots && \\ 
& \ddots & \ddots & \\ 
& & 1 & 0
\end{array}
\right).
\end{align*}
The pair $(\K_r,\Phi(q))$ is called a cyclic Higgs bundle (cf. \cite{Bar1, Hit1, Hit2}). The above cyclic Higgs bundle is an example of cyclotomic Higgs bundles introduced in \cite{Sim3}. For each Hermitian metric $h$ on $\K_r$, we associate a connection $D_q(h)$ defined as follows:
\begin{align*}
D_q(h)\coloneqq \nabla^h + \Phi(q) + \Phi(q)^{\ast h}.
\end{align*}
A Hermitian metric on $(\K_r,\Phi(q))\rightarrow X$ is called a {\it harmonic metric} if the connection $D_q(h)$ is a flat connection. A harmonic metric $h$ yields a harmonic map $\hat{h}: \widetilde{X}\rightarrow \SL(r,\C)/\SU(r)$, where $\widetilde{X}$ is the universal covering space. If $h$ is real, then $\hat{h}$ maps into $\SL(r,\R)/\SO(r)$. Suppose that $X$ is a compact Riemann surface of genus at least $2$. Then the Higgs bundle $(\K_r,\Phi(q))$ is stable for any $q\in H^0(K_X^r)$, and thus there uniquely exists a harmonic metric $h$ on $(\K_r,\Phi(q))$ such that $\det(h)=1$. From the uniqueness of the metric and the symmetry of the Higgs field (cf. \cite{Bar1, Hit2}), the harmonic metric $h$ is diagonal and real. To construct a non-diagonal harmonic metric (resp. a non-real harmonic metric), one may, for example, use a solution to the Dirichlet problem for the Hitchin equations with a non-diagonal (resp. non-real) Hermitian metric prescribed on the boundary (cf. \cite{Don1, LM2}).
Let $h=(h_1,\dots, h_r)$ be a diagonal harmonic metric on a cyclic Higgs bundle $(\K_r,\Phi(q))\rightarrow X$. As mentioned in the previous subsection we obtain $r-1$ Hermitian metrics $H_1,\dots, H_{r-1}$ on $T_X\rightarrow X$ defined as
\begin{align*}
H_j\coloneqq h_j^{-1}\otimes h_{j+1} \ \text{for each $j=1,\dots,r -1$}.
\end{align*}
We also obtain a Hermitian metric $H_r$ on $T_X\rightarrow X$ which is degenerate at the zeros of $q$ as follows:
\begin{align*}
H_r\coloneqq h_r^{-1}\otimes h_1 \otimes h_q,
\end{align*}
where $h_q$ is a Hermitian metric on $K_X^{-r}\rightarrow X$ that is degenerate at the zeros of $q$ defined as follows:
\begin{align}
(h_q)_x(u,v)\coloneqq \langle q_x,u\rangle\overline{\langle q_x,v\rangle} \ \text{for $x\in X$, $u,v\in (K_X^{-r})_x$}. \label{q}
\end{align}
The connection $D_q(h)$ is flat if and only if the Hermitian metric $h$ satisfies the following elliptic equation, called the Hitchin equation:
\begin{align*}
F_h + [\Phi(q)\wedge\Phi(q)^{\ast h}] = 0.
\end{align*}
By using $H_1,\dots, H_r$, we can describe the Hitchin equation as follows: 
\begin{align}
\inum F_{h_j} + \vol(H_{j-1}) - \vol(H_j) = 0 \ \text{for $j=1,\dots, r-1$}, \label{F}
\end{align}
where for each $j=1,\dots, r$, $\vol(H_j)$ is the volume form of the metric $H_j$ which is locally described as $\vol(H_j)=\inum H_j(\frac{\partial}{\partial z}, \frac{\partial}{\partial z})\ dz\wedge d\bar{z}$, and $H_0$ is $H_r$. Equation (\ref{F}) is a kind of Toda equation (cf. \cite{AF1, Bar1, GH1, GL1, Moc0, Moc1}). By making a variable change from $h=(h_1,\dots, h_r)$ to $H_1,\dots, H_{r-1}$, the Hitchin equation becomes the following:
\begin{align}
\inum F_{H_j} + 2\vol(H_j) - \vol(H_{j-1}) - \vol(H_{j+1}) = 0 \ \text{for $j=1,\dots, r-1$.} \label{cyc}
\end{align}
In the case where $r=2$ and $q=0$, the Hitchin equation for the cyclic Higgs bundle is equivalent to the following equation:
\begin{align}
\inum F_H + 2\vol(H) = 0, \label{H}
\end{align}
where the solution $H$ to equation (\ref{H}) is a Hermitian metric $H$ on $T_X\rightarrow X$. The K\"ahler metric induced by $H$ is a constant negative Gaussian curvature K\"ahler metric (cf. \cite{Hit1}). 
We remark on the following two points:
\begin{itemize}
\item If $q$ has no zeros, then $H_1=\cdots=H_{r-1}=h_q^{1/r}$ is a solution to equation (\ref{cyc}).
\item Suppose that $q=0$ and that there is a solution $H$ to equation (\ref{H}). We define positive constants $\lambda_1,\dots, \lambda_{r-1}$ as follows: 
\begin{align}
\lambda_j\coloneqq 2\sum_{k=1}^{r-1}(\Lambda_{r-1}^{-1})_{jk} = j(r-j), \label{lambda}
\end{align}
where $\Lambda_{r-1}^{-1}$ is the inverse matrix of the Cartan matrix of type $A_{r-1}$. Then $H_j=\lambda_j H\ (j=1,\dots, r-1)$ is a solution to equation (\ref{cyc}).
\end{itemize}
Li--Mochizuki \cite{LM1} established the following result:
\begin{theo}[\cite{LM1}]{\it Suppose that $q$ is a non-zero holomorphic section unless $X$ is hyperbolic. Then there always exists a unique complete harmonic metric $h$ on $(\K_r,\Phi(q))\rightarrow X$ satisfying $\det(h)=1$. Moreover, the complete harmonic metric $h$ is real.
}
\end{theo}
In addition to \cite{LM1}, we refer the reader to \cite{DL3, DW1, GL1, LT1, LTW1, Li1, LM2, Moc0, Moc1, Nie1, Wan1, WA1} for the works related to complete harmonic metrics. Also, Dai--Li \cite{DL1, DL2} and Li--Mochizuki \cite{LM1} established the following:
\begin{theo}[\cite{DL1,DL2, LM1}]\label{estimate} 
{\it Let $q$ be an $r$-differential that is not identically zero and is not induced by a flat complete K\"ahler metric. Let $h$ be the diagonal complete harmonic metric on the cyclic Higgs bundle $(\K_r,\Phi(q))$ satisfying $\det(h)=1$. Then the following estimate holds:
\begin{align}
&\frac{\lambda_{j-1}}{\lambda_j}=\frac{(j-1)(r-j+1)}{j(r-j)} < H_{j-1}\otimes H_j^{-1} < 1 \ \text{for all $2\leq j\leq [r/2]$},\label{r/2}\\
& H_r\otimes H_1^{-1} < 1.
\end{align}
}
\end{theo}
The above theorem states that the difference between the adjacent metrics $H_{j-1}\otimes H_j^{-1}$ can be evaluated using the metrics for the above two extremal cases. For theorems on harmonic maps which are derived from the above estimate, see \cite{DL1, DL2, DL3, LM1}.

\begin{rem} When $r=2$ and $X$ is a compact hyperbolic Riemann surface, the inequality $H_2\otimes H_1^{-1}<1$ was shown by Hitchin in the proof of \cite[Theorem 11.2]{Hit1}.
As a consequence, it holds that $q+\omega_1+\omega_2+\bar{q}$ is positive-definite, where $q$ is considered to be a $(2,0)$-form, and $\omega_j \ (j=1,2)$ is a $(1,1)$-form induced by $H_j$ (see also \cite[Theorem 3.1]{Wol1}). 
\end{rem}
\begin{rem} Since \cite{Wol1}, how the accompanying harmonic metric depends on $t$ when the holomorphic $r$- differential $q$ is scale-transformed to $tq$ by the parameter $t$ has been well studied (cf. \cite{CL1, DL2, DW1, Lof1, LoTW1, Sak1, Wol1}), but there seems to be much less research investigating how the location of zeros affects the global properties of a harmonic metric. Note, however, that the local behavior of the harmonic metric on the cyclic Higgs bundle around the zeros of $q$, especially for $r=2,3$, has been well studied. See \cite[Section 1.8]{DW1} and the references cited therein. We also remark on \cite{Li1, LM1, LM2, Nie1} for studies on the relationship between the completeness of the metric induced by $q$ outside the zero sets and the uniqueness of the harmonic metric. There has also been a great deal of research into the zeros of holomorphic differentials, forgetting about harmonic metrics, especially in the case of $r=2,3$. See, for example, \cite{BS1, DW1, Kus1, Mas1} and the references cited therein.
\end{rem}
\subsection{More general subharmonic weight functions}\label{2.2}
For each $q\in H^0(K_X^r)$, the metric $h_q$ induces a singular metric (cf. \cite{GZ1}) $h_q^{-1/r}$ that diverges at the zeros of $q$. We can define a curvature for $h_q^{-1/r}$ in the sense of currents, which is semipositive and has support at the zero points of $q$. We fix a smooth metric $h_\refe$ and denote by $e^{-\phi_q}h_\refe$ the metric $h_q^{-1/r}$, where the weight function $\phi_q$ is defined as $\phi_q\coloneqq \frac{1}{r}\log|q|_{h_\refe}^2$. If we choose a local flat reference metric, then the weight function defines a local subharmonic function. Also, although we will not go into details here, there is a way to identify the weight function with a global psh function on the total space of the dual line bundle excluding zero points, see \cite[Section 2.1]{BB1}. In the sense described above, we call the function $\phi_q$ a {\it subharmonic weight function}, although this is a bit of an abuse of terminology. We consider a more general singular metric $e^{-\varphi}h_\refe$ with semipositive curvature, where the weight function $\varphi$ is a function that is locally a sum of a subharmonic function and a smooth function. For each $q_N\in H^0(K_X^N)$, let $\phi_{q_N}$ be defined as $\frac{1}{N}\log|q_N|_{h_\refe}^2$. Then $e^{-\phi_{q_N}}h_\refe$ is a singular metric with semipositive curvature. Moreover, any subharmonic weight function can be approximated, at least in the $L^1_\loc$ sense, by a sequence $(\phi_{q_N})_{N\in\N}$, where $q_N\in H^0(K_X^N)$ (cf. \cite{GZ1}). In \cite{Miy2, Miy3}, for each semipositive singular metric $e^{-\varphi}h_\refe$ and each $r\geq 2$, the following equation for a diagonal metric $h=(h_1,\dots, h_r)$ on $\K_r\rightarrow X$ was introduced:
\begin{align}
\inum F_{h_j} + \vol(H_{j-1}) - \vol(H_j) = 0 \ \text{for $j=1,\dots, r-1$}, \label{phi}
\end{align}
where $H_1,\dots, H_{r-1}$ are defined as $H_j\coloneqq h_j^{-1}\otimes h_{j+1}$ for each $j=1,\dots, r-1$, and $H_r=H_0$ is defined as follows:
\begin{align*}
H_r = H_0 &\coloneqq h_r^{-1}\otimes h_1 \otimes (e^{-\varphi}h_\refe)^{-r} \\ 
&= H_1^{-1}\otimes \cdots \otimes H_{r-1}^{-1} \otimes (e^{-\varphi}h_\refe)^{-r}.
\end{align*}
For the case where $\varphi=\phi_q=\frac{1}{r}\log|q|_{h_\refe}^2$, equation (\ref{phi}) becomes the Hitchin equation for the cyclic Higgs bundle $(\K_r,\Phi(q))$. 
For $\varphi=\phi_{q_N}$, a solution of equation (\ref{phi}) pulls back to a harmonic metric on the bundle $\pi^\ast\K_r$ over a suitable cyclic branched covering $\pi:Y\to X$ on which $q_N^{1/N}$ is single-valued. The corresponding Higgs field is defined using this root and the natural map $\pi^\ast K_X\to K_Y$ (see \cite[Section 2]{Miy3}). As mentioned above, more general $\varphi$ can be approximated by a sequence $(\phi_{q_N})_{N\in\N}$ at least in the $L^1_{\loc}$-sense (cf. \cite{GZ1}); therefore, equation (\ref{phi}) can be considered to be an equation obtained as the limit when the covering degree increases infinitely or the number of zeros increases infinitely. By making a variable change from $h=(h_1,\dots, h_r)$ to $H_1,\dots, H_{r-1}$, the equation becomes the following:
\begin{align}
\inum F_{H_j} + 2\vol(H_j) - \vol(H_{j-1}) - \vol(H_{j+1}) = 0 \ \text{for $j=1,\dots, r-1$.}\label{cyc2}
\end{align}

We remark on the following two points:
\begin{itemize}
\item If $e^{-\varphi}h_\refe$ is flat, then $H_1=\cdots=H_{r-1}=(e^{-\varphi}h_\refe)^{-1}$ is a solution to equation (\ref{cyc2}).
\item Suppose that $\varphi=-\infty$. Then equation (\ref{cyc2}) is the same as equation (\ref{cyc}) for the case where $q=0$, and thus for a solution $H$ to equation (\ref{H}), $H_j=\lambda_j H\ (j=1,\dots, r-1)$ is a solution to equation (\ref{cyc2}), where $\lambda_1,\dots, \lambda_{r-1}$ are constants defined in (\ref{lambda}).
\end{itemize}
We describe the existence, uniqueness, and approximation of complete solutions to equation (\ref{phi}). 
Then the following results hold.

\begin{theo}[Existence, \cite{Miy5, Miy6}]\label{ex}
{\it
Let $\varphi$ be a subharmonic weight function on $X$. Assume that $\varphi$ is not identically $-\infty$, unless $X$ is hyperbolic. Then there exists a real complete solution $h=(h_1,\dots,h_r)$ associated with $\varphi$.
}
\end{theo}

\begin{theo}[Uniqueness, \cite{Miy5, Miy6}]\label{un}
{\it
Let $h$ and $h^\prime$ be complete solutions associated with $\varphi$ such that $\det(h)=\det(h^\prime)=1$.
Then $h=h^\prime$. 
}
\end{theo}

\begin{theo}[Approximation, \cite{Miy5, Miy6}]\label{ap}
{\it
After passing to the universal covering space, the unique complete solution with determinant 1 can be approximated by smooth complete solutions in the following sense. Let $(\varphi_\epsilon)_{\epsilon>0}$ be a mollification of the pullback of $\varphi$ to the universal covering space. Then the corresponding family of complete solutions $(h_\epsilon)_{\epsilon>0}$, normalized by $\det(h_\epsilon)=1$,
converges, as $\epsilon\searrow 0$, to the pullback of the complete solution associated with $\varphi$ such that $\det(h)=1$.
}
\end{theo}

\begin{rem} 
The author is not aware of any work other than \cite{Miy2, Miy3} on harmonic metrics on cyclic ramified coverings considered in \cite{Miy2, Miy3}, but would like to remark on the following: The Hitchin equation for rank $2$ cyclic Higgs bundles has similarity to the PDE appearing in the Ginzburg--Landau model in physics (cf. \cite[Section 5]{DW1}, \cite[Section 11]{Hit1}, \cite{JT1}), and the Ginzburg--Landau model has ties to the asymptotic behavior of empirical measures (cf. \cite{Ser1, Ser2}).
\end{rem}

\section{Shannon entropy and Schur-concave functions}\label{Shannon}

In this section, we review Shannon entropy and, more generally, Schur-concave functions. We begin by recalling the definition of Shannon entropy.

\begin{defi}[\cite{Sh1}]
Let $r\geq 2$ be an integer, and let $p_0,p_1,\dots,p_{r-1}$ be nonnegative real numbers satisfying
\begin{align*}
p_0+p_1+\cdots+p_{r-1}=1.
\end{align*}
The {\it Shannon entropy} of $p_0,p_1,\dots,p_{r-1}$ is defined by
\begin{align*}
S(p_0,p_1,\dots,p_{r-1})
\coloneqq
-\sum_{j=0}^{r-1}p_j\log p_j,
\end{align*}
where we use the convention $0\log 0=0$.
\end{defi}
Shannon entropy is one of the most important examples of {\it Schur-concave functions}. We next review the general theory of Schur-concave functions. Let $r\geq 2$ be an integer. For each $x=(x_1,\dots,x_r)\in \R^r$, we denote the components of $x$ arranged in descending order by
\begin{align*}
x_{[1]}\geq x_{[2]}\geq \cdots \geq x_{[r]}.
\end{align*}
We introduce the following notion:
\begin{defi}[cf. \cite{MOA1}]
For $x=(x_1,\dots,x_r),y=(y_1,\dots,y_r)\in \R^r$, $y$ is said to {\it majorize} $x$ (or $x$ is said to be {\it majorized by} $y$), in symbols $y\succ x$, if
\begin{align*}
\sum_{i=1}^m x_{[i]}\leq \sum_{i=1}^m y_{[i]}
\quad \text{for $m=1,2,\dots,r-1$},
\end{align*}
and
\begin{align*}
\sum_{i=1}^r x_i=\sum_{i=1}^r y_i.
\end{align*}
\end{defi}
We then define the notion of Schur-concave functions:
\begin{defi}[cf. \cite{MOA1}]
Let $A\subseteq \R^r$ be a permutation-invariant subset. A real-valued function $f$ defined on $A$ is said to be {\it Schur-concave} on $A$ if
\begin{align}
f(y)\leq f(x) \label{f}
\end{align}
for all $x,y\in A$ such that $y\succ x$. If equality in (\ref{f}) holds only when $x$ is a permutation of $y$, then $f$ is said to be {\it strictly Schur-concave} on $A$. Similarly, a function $f$ is said to be {\it Schur-convex} (resp. {\it strictly Schur-convex}) if $-f$ is Schur-concave (resp. strictly Schur-concave).
\end{defi}
The following Schur--Ostrowski criterion is useful and well known:
\begin{prop}[cf. \cite{MOA1}]\label{SO}
{\it
Let $I\subseteq \R$ be an open interval and let $f:I^r\rightarrow \R$ be a permutation-invariant $C^1$ function. Then $f$ is Schur-concave if and only if
\begin{align}
(x_i-x_j)
\left(
\frac{\partial f}{\partial x_i}(x)-
\frac{\partial f}{\partial x_j}(x)
\right)
\leq 0\label{xixj}
\end{align}
for all $x\in I^r$ and all $i,j=1,\dots,r$. Moreover, if
\begin{align}
(x_i-x_j)
\left(
\frac{\partial f}{\partial x_i}(x)-
\frac{\partial f}{\partial x_j}(x)
\right)
< 0\label{xixj<}
\end{align}
for every $x\in I^r$ and all $i,j$ such that $x_i\neq x_j$, then $f$ is strictly Schur-concave.
}
\end{prop}
For the proof, we refer the reader to \cite[pp.~54--59]{MOA1}. The following is a variant of Proposition \ref{SO} for functions on probability simplices:
\begin{prop}
{\it
Let $\Delta_{r-1}\subseteq \R^r$ be the probability simplex
\begin{align*}
\Delta_{r-1}\coloneqq
\{(p_1,\dots,p_r)\in\R^r
\mid p_j\geq 0 \text{ for } j=1,\dots,r,\
p_1+\cdots+p_r=1\}.
\end{align*}
Let $f:\Delta_{r-1}\rightarrow\R$ be a permutation-invariant continuous function that is $C^1$ on the relative interior of $\Delta_{r-1}$. Then $f$ is Schur-concave if and only if
\begin{align*}
(p_i-p_j)(df)_p(u_i-u_j)\leq0
\end{align*}
for every $p=(p_1,\dots,p_r)$ in the relative interior of $\Delta_{r-1}$ and all $i,j=1,\dots,r$, where $u_1,\dots,u_r$ are the standard basis vectors of $\R^r$ and $(df)_p$ denotes the differential of $f$ at $p$. Moreover, if
\begin{align*}
(p_i-p_j)(df)_p(u_i-u_j)<0 
\end{align*}
whenever $p_i\neq p_j$, then $f$ is strictly Schur-concave.
}
\end{prop}
The proof is a straightforward adaptation of the proof of Proposition \ref{SO}, and we omit it. Using the above Schur--Ostrowski criterion, we can check the strict Schur-concavity of the Shannon entropy as follows. Take arbitrary $1\leq i<j\leq r$. For every $p=(p_1,\dots,p_r)$ in the relative interior of $\Delta_{r-1}$ such that $p_i\neq p_j$, we have
\begin{align*}
(p_i-p_j)
\left.\frac{d}{dt}\right|_{t=0}
S(p_1,\dots,p_i+t,\dots,p_j-t,\dots,p_r)
&=
-(p_i-p_j)(\log p_i-\log p_j)<0.
\end{align*}
We refer the reader to \cite{Lei1, MOA1, Niel1} for more examples of Schur-concave functions.

\section{The Cheng--Yau maximum principle and the mean value inequality}\label{CYMMVI}
This section briefly reviews the Cheng--Yau maximum principle \cite{CY1} (see also \cite[Section 3.1.2]{LM1}) and the mean value inequality \cite[Lemma 2.5]{CT1} (see also \cite[Section 2]{DL3}). The following theorem is known as the Cheng--Yau maximum principle, which played a crucial role in \cite{LM1}:
\begin{theo}[\cite{CY1}]\label{C-Ym}{\it Let $(M, g_M)$ be a complete Riemannian manifold with Ricci curvature bounded from below. Let $u$ be a real-valued $C^2$ function on $M$ satisfying $\Delta_{g_M} u \geq f(u)$, where $f:\R\rightarrow \R$ is a function, and where we denote by $\Delta_{g_M} = -d^\ast d$ the negative Laplacian. Suppose that there exists a continuous positive function $g:[a,\infty)\rightarrow \R_{>0}$ such that
\begin{enumerate}[(i)]
\item $g$ is non-decreasing;
\item $\liminf_{t\to\infty}\frac{f(t)}{g(t)}>0$;
\item $\int_b^\infty(\int_a^tg(\tau)d\tau)^{-1/2}dt<\infty$ for some $b\geq a$.
\end{enumerate}
Then the function $u$ is bounded above. Moreover, if $f$ is lower semicontinuous, then $f(\sup_M u)\leq 0$.
}
\end{theo}
We recall the following mean-value inequality from \cite[Proposition 2.1]{DL3}, which is based on \cite[Lemma 2.5]{CT1}.
\begin{theo}[\cite{DL3}]\label{MVI}{\it Let $(M,g_M)$ be a complete Riemannian manifold of dimension $n$. Consider the differential inequality 
\begin{align}
\Delta_{g_M}u\leq c u \label{cuagain}
\end{align}
with some positive constant $c$. Let $x_0\in M$. We denote by $B_{g_M}(x_0,r)$ the open geodesic ball of radius $r$ centered at $x_0$, and by $\Vol(B_{g_M}(x_0, r))$ the volume of the geodesic ball. We fix a positive constant $R_0$ and suppose the following conditions hold for $x_0$ and $R_0$:
\begin{enumerate}[(i)]
\item The Poincar\'e and the Sobolev inequalities hold for functions supported on $B_{g_M}(x_0, R_0)$ with constants $c_p$ and $c_s$;
\item There exists a positive constant $c_2$ such that $\Vol(B_{g_M}(x_0,r))\leq c_2r^n$ for all $r\leq R_0$.
\end{enumerate}
Then there exist positive constants $p_0$ and $C$ depending only on $n, c, R_0, c_2, c_p, c_s$ such that for any continuous nonnegative $W^{1,2}$-function $u$ satisfying (\ref{cuagain}) on $B_{g_M}(x_0, R_0)$ and any $0<p<p_0$ the following inequality holds:
\begin{align*}
\inf_{x\in B_{g_M}(x_0,R_0/4)} u(x)\geq C\left(\int_{B_{g_M}(x_0,R_0/2)}u^pd\mu_{g_M}\right)^{1/p},
\end{align*}
where $d\mu_{g_M}$ is the volume measure. In particular, there exist constants $C>0$ and $0<p<1$ such that 
\begin{align}
u(x_0)\geq C\left(\int_{B_{g_M}(x_0,R_0/2)}u^pd\mu_{g_M}\right)^{1/p}. \label{itten}
\end{align}
}
\end{theo}
\begin{rem} In the statements in \cite{CT1, DL3}, the choice of a representative
of a $W^{1,2}$-function is not explicitly specified. In particular,
without specifying a representative, a pointwise estimate such as
\eqref{itten} is not well-defined for an equivalence class of
$W^{1,2}$-functions. We therefore impose the additional assumption
that $u$ is continuous in the theorem above. In \cite{DL3}, the
mean-value inequality is applied to a smooth function in the proof
of Lemma 2.2, so this issue does not arise in that application.
\end{rem}

For later use, we prove the following extension of Theorem \ref{MVI}.
\begin{lemm}\label{wMV}{\it Let $(M,g_M)$, $x_0\in M$, and $R_0>0$ satisfy the assumptions of Theorem \ref{MVI}. Let $(u_k)_{k\in\N}$ be a sequence of continuous nonnegative $W^{1,2}$-functions on $B_{g_M}(x_0,R_0)$ satisfying
\begin{align}
\Delta_{g_M}u_k\leq c u_k
\label{cuk}
\end{align}
for some positive constant $c$ independent of $k$. Let $p_0$ and $C$ be the positive constants given by Theorem \ref{MVI}. Suppose that there exists a measurable function $u: B_{g_M}(x_0,R_0)\rightarrow \R_{\geq 0}$ such that $(u_k)_{k\in\N}$ converges to $u$ almost everywhere on $B_{g_M}(x_0,R_0)$. Then, for every $0<p<p_0$, we have $u\in L^p(B_{g_M}(x_0,R_0/2))$ and
\begin{align}
\inf_{x\in B_{g_M}^{(u)}(x_0,R_0/4)}u(x)
\geq
C
\left(
\int_{B_{g_M}(x_0,R_0/2)}u^p d\mu_{g_M}
\right)^{1/p}, \label{essinf}
\end{align}
where the set $B_{g_M}^{(u)}(x_0,R_0/4)$ is defined by 
\begin{align*}
B_{g_M}^{(u)}(x_0,R_0/4)\coloneqq \{x\in B_{g_M}(x_0,R_0/4)\mid \text{$(u_k(x))_{k\in\N}$ converges to $u(x)$}\}.
\end{align*}
In particular, for each $x\in B_{g_M}^{(u)}(x_0,R_0/4)$, we have
\begin{align*}
u(x)\geq C
\left(
\int_{B_{g_M}(x_0,R_0/2)}u^p d\mu_{g_M}
\right)^{1/p}.
\end{align*}
}
\end{lemm}
\begin{proof}
Fix $0<p<p_0$ and $x\in B^{(u)}_{g_M}(x_0,R_0/4)$. By Fatou's lemma and Theorem \ref{MVI}, we obtain
\begin{align*}
C\left(\int_{B_{g_M}(x_0,R_0/2)}u^pd\mu_{g_M}\right)^{1/p}&\leq C\liminf_{k\to\infty}\left(\int_{B_{g_M}(x_0,R_0/2)}u_k^pd\mu_{g_M}\right)^{1/p} \\
&\leq \liminf_{k\to\infty}u_k(x) \\
&=u(x).
\end{align*}
Since $u(x)<\infty$, we see that $u\in L^p(B_{g_M}(x_0,R_0/2))$. Also, by taking the infimum, we obtain (\ref{essinf}).
\end{proof}
\begin{rem}In the applications in \cite{DL3}, it is essential that the constants \(c_s\), \(c_p\), and \(c_2\) can be chosen uniformly, independently of the point \(x\). In the present work, however, we use the mean-value inequality only locally and do not require such uniformity.
\end{rem}

\section{Definition of entropy}\label{4}
\subsection{Definition}\label{4.1}
Let $h=(h_1,\dots, h_r)$ be a solution to equation (\ref{phi}) and let $H_1,\dots, H_r$ be the Hermitian metrics constructed from $h$ and $e^{-\varphi}h_\refe$ as in Section \ref{2.2}. Then we define a function, which we call entropy, as follows:
\begin{defi}\label{entropy2} 
For each $j=0,1,\dots,r-1$ and a nonzero real number $\beta$, 
we define a nonnegative function
$p_j(r,\beta,\varphi):X\to[0,1]$ as follows.
For each $x\in X$, if $(\beta\varphi)(x)$ does not take the value $\infty$,
which means that either $\beta>0$, or $\beta<0$ and
$x\notin\varphi^{-1}(-\infty)$, we define
\begin{align*}
p_j(r,\beta,\varphi)(x)\coloneqq
\frac{\vol(H_j)^\beta}{\sum_{j=0}^{r-1}\vol(H_j)^\beta}(x) \ \text{for each $j=0,\dots, r-1$,}
\end{align*}
where $\vol(H_j)^\beta/\sum_{j=0}^{r-1}\vol(H_j)^\beta$ is understood to be
$(\vol(H_j)/\vol(H_\refe))^\beta/\sum_{j=0}^{r-1}(\vol(H_j)/\vol(H_\refe))^\beta$
for some reference metric $H_\refe$, which does not affect
$p_j(r,\beta,\varphi)$.
If $(\beta\varphi)(x)$ takes the value $\infty$, we set
$p_0(r,\beta,\varphi)(x)=0$ and define
\begin{align*}
p_j(r,\beta,\varphi)(x)\coloneqq
\frac{\vol(H_j)^\beta}{\sum_{j=1}^{r-1}\vol(H_j)^\beta}(x) \ \text{for each $j=1,\dots, r-1$.}
\end{align*}
We define the {\it entropy} by
\begin{align*}
S(r,\beta,\varphi)\coloneqq
-\sum_{j=0}^{r-1} p_j(r,\beta,\varphi)\log p_j(r,\beta,\varphi).
\end{align*}
\end{defi}
\begin{rem}
In this paper, we consider only those solutions to equation (\ref{phi}) used to define the entropy that are complete and having determinant 1. Strictly speaking, the entropy on an open Riemann surface depends not only on the parameters $r$, $\beta$, and $\varphi$, as indicated in the notation, but also on the boundary conditions satisfied by the solution to equation (\ref{phi}).
\end{rem}
\begin{rem} The constant $\beta$ is modelled after the inverse temperature in the canonical ensemble in statistical mechanics (cf. \cite{LL1}). We will further pursue the analogy with the canonical ensemble in the subsequent paper \cite{Miy4}.
\end{rem}
\subsection{Two extremal cases}\label{4.2}

We remarked in Section \ref{2.2} that there are two explicit complete solutions to equation (\ref{phi}) constructed from hyperbolic and Euclidean metrics. In this subsection, we note that the entropy introduced in the previous subsection is constant for these two solutions. We first consider the case where $X$ is a parabolic Riemann surface and the weight function $\varphi=\varphi_{\rm Euc}$ is induced by the Euclidean metric. In this case, we have $H_1=\cdots=H_r=e^{\varphi_{\rm Euc}}h_\refe^{-1}$, and thus we obtain
\begin{align*}
S(r,\beta,\varphi_{\rm Euc})=\log r.
\end{align*}
We next consider the case where the Riemann surface $X$ is hyperbolic and the weight function is identically $-\infty$. Let $H$ be the Hermitian metric on $K_X^{-1}$ induced by the hyperbolic metric. In this case, we have
\begin{align*}
H_j=\lambda_j H,
\end{align*}
where $\lambda_j$ is defined by $\lambda_j\coloneqq j(r-j)$.
Then we obtain
\begin{align*}
&S(r,\beta,-\infty)=-\sum_{j=1}^{r-1}\hat{p}_{j,\beta}\log \hat{p}_{j,\beta},\\
&\hat{p}_{j,\beta}\coloneqq \frac{\lambda_j^\beta}{\sum_{j=1}^{r-1}\lambda_j^\beta}.
\end{align*}
We set
\begin{align*}
S_{r,\beta}\coloneqq S(r,\beta,-\infty).
\end{align*}

\section{Proof}\label{6}
\subsection{Proof of Theorem \ref{main theorem 1}}
We first prove Theorem \ref{main theorem 1}. The following proposition is an extension of the estimate established by Dai--Li \cite{DL1, DL2} and Li--Mochizuki \cite{LM1} to more general subharmonic weight functions:
\begin{prop}\label{r/2}
{\it Suppose that $\varphi$ is not identically $-\infty$ and that $e^{-\varphi}h_\refe$ is not induced by a flat complete K\"ahler metric. Let $H_1,\dots, H_r$ be Hermitian metrics on $K_X^{-1}\rightarrow X$ constructed from $h$ and $e^{-\varphi}h_\refe$ as in Section \ref{2.2}. Then the following estimate holds:
\begin{align}
&\frac{\lambda_{j-1}}{\lambda_j}=\frac{(j-1)(r-j+1)}{j(r-j)}<H_{j-1}\otimes H_j^{-1}<1 \ \text{for all $2\leq j\leq [r/2]$},\label{Gauss}\\
& H_r\otimes H_1^{-1}<1. \label{Euler}
\end{align}
}
\end{prop}
We first prove the above Proposition \ref{r/2} for the case where $\varphi$ is smooth by following \cite[Theorem 4.4]{LM1}, and then we prove the assertion by using the approximation of the complete solution and the mean value inequality (see Section \ref{CYMMVI}). We set $n\coloneqq [r/2]$ and $\delta\coloneqq r-2n$. Let $H_1,\dots, H_r$ be the Hermitian metrics on $K_X^{-1}\rightarrow X$ constructed by using a complete solution $h=(h_1,\dots, h_r)$ to equation (\ref{phi}) and $e^{-\varphi}h_\refe$. We define $\sigma_1,\dots, \sigma_n$ as follows:
\begin{align*}
\sigma_j\coloneqq \log(H_{j-1}\otimes H_j^{-1}) \ \text{for $j=1,2,\dots, n$.}
\end{align*}
When $n=1$ (i.e., $r=2,3$), we use the convention $\sigma_2:=-\sigma_1$ throughout the proof, so that the arguments below apply without modification. Let $r\geq 4$. Then we also define $\sigma^\prime_1,\dots, \sigma_{n-1}^\prime$ as follows:
\begin{align*}
\sigma_j^\prime\coloneqq \log(H_j^{-1}\otimes H_{j+1}) \ \text{for $j=1,\dots, n-1$}.
\end{align*}
By direct calculation, we can check that the following holds:
\begin{align}
&\inum\partial\bar{\partial}\sigma_1\geq 3\vol(H_0)-4\vol(H_1)+\vol(H_2), \label{sigma} \\
&\inum \partial\bar{\partial}\sigma_j =3\vol(H_{j-1})-3\vol(H_j)-\vol(H_{j-2})+\vol(H_{j+1})\ \text{for $j=2,\dots, n-1$}, \label{sigma2}\\
&\inum \partial\bar{\partial} \sigma_n=(4-\delta)\vol(H_{n-1})-(3-\delta)\vol(H_n)-\vol(H_{n-2}), \label{sigma3}\\
&\inum \partial \bar{\partial} \sigma^\prime_j=-3\vol(H_j)+3\vol(H_{j+1})+\vol(H_{j-1})-\vol(H_{j+2})\ \text{for $j=1,\dots, n-2$}, \label{sigma4}\\
&\inum \partial \bar{\partial} \sigma^\prime_{n-1}=(3-\delta)\vol(H_n)-(4-\delta)\vol(H_{n-1})+\vol(H_{n-2}). \label{sigma5}
\end{align}
Note that we have used in inequality (\ref{sigma}) the assumption that the metric $e^{-\varphi}h_\refe$ is semipositive. We also note that when the regularity of $\varphi$ is not high enough, inequality (\ref{sigma}), equality (\ref{sigma2}) for $j=2$, and equality (\ref{sigma4}) hold only in the weak sense. 
From (\ref{sigma}), (\ref{sigma2}), and (\ref{sigma3}), we obtain the following:
\begin{align}
&\inum\Lambda_{H_1}\partial \bar{\partial} \sigma_1\geq 3e^{\sigma_1}-4+e^{-\sigma_2}, \label{Hsigma}\\
&\inum \Lambda_{H_j}\partial\bar{\partial}\sigma_j =3e^{\sigma_j}-3-e^{\sigma_j+\sigma_{j-1}}+e^{-\sigma_{j+1}}\ \text{for $j=2,\dots, n-1$}, \label{Hsigma2} \\
&\inum \Lambda_{H_n}\partial\bar{\partial} \sigma_n=(4-\delta)e^{\sigma_n}-(3-\delta)-e^{\sigma_{n-1}+\sigma_n}. \label{Hsigma3}
\end{align}
From (\ref{sigma4}) and (\ref{sigma5}), we have
\begin{align}
&\inum \Lambda_{H_1}\partial \bar{\partial} \sigma^\prime_1\geq -3+3e^{\sigma_1^\prime}-e^{\sigma_1^\prime+\sigma_2^\prime}, \label{Hs}\\
&\inum \Lambda_{H_j}\partial \bar{\partial} \sigma^\prime_j=-3+3e^{\sigma_j^\prime}+e^{-\sigma_{j-1}^\prime}-e^{\sigma_j^\prime+\sigma_{j+1}^\prime}\ \text{$j=2,\dots, n-2$}, \label{Hsigma4}\\
&\inum \Lambda_{H_{n-1}}\partial \bar{\partial} \sigma^\prime_{n-1}=(3-\delta)e^{\sigma_{n-1}^\prime}-(4-\delta)+e^{-\sigma_{n-2}^\prime}. \label{Hsigma5}
\end{align}
We introduce the following notation:
\begin{align*}
&M_j\coloneqq \sup_Xe^{\sigma_j}\ \text{for $j=1,\dots, n$}, \\
&M_j^\prime\coloneqq \sup_Xe^{\sigma_j^\prime} \ \text{for $j=1,\dots, n-1$}, \\
&B_1\coloneqq 2(1-M_1^{-1}), \\
&B_j\coloneqq 2-M_j^{-1}-M_{j-1} \ \text{for $j=2,\dots, n$}, \\
&B_{n+1}\coloneqq (2-\delta)(1-M_n), \\
&B_0^\prime\coloneqq 2-M_1^\prime, \\
&B_j^\prime\coloneqq 2-M_j^{\prime-1}-M_{j+1}^\prime \ \text{for $j=1,\dots, n-2$}, \\
&B_{n-1}^\prime\coloneqq (2-\delta)(1-M_{n-1}^{\prime-1}).
\end{align*}
Following \cite[Section 4]{LM1}, with a slight variation of the argument, we first show the following:
\begin{lemm}\label{MMBB}{\it 
Suppose that $\varphi$ is smooth. Then the following holds:
\begin{align}
&M_j \leq \frac{4j}{2j+1}-\frac{2j-1}{2j+1}M_{j+1}^{-1}\ \text{for $j=1,\dots, n-1$}, \label{M}\\
&M_j B_j\leq B_{j+1}\ \text{for $j=1,\dots, n$}, \label{MB}\\
&M_{n-j}^\prime \leq \frac{4j-(2j-1)\delta}{2j+1-j\delta}-\frac{2j-1-(j-1)\delta}{2j+1-j\delta}M_{n-(j+1)}^{\prime-1} \ \text{for $j=1,\dots, n-2$}, \label{Mp}\\
&M_{n-j}^\prime B_{n-j}^\prime \leq B_{n-(j+1)}^\prime\ \text{for all $j=1,\dots, n-1$}. \label{MpBp}
\end{align}
}
\end{lemm}
\begin{proof}
For each Hermitian metric $H$ on $K_X^{-1}\rightarrow X$, we denote by $\Lambda_{H}$ the contraction operator of the K\"ahler metric induced by $H$. We first prove (\ref{M}) and (\ref{MB}) by induction on $j$. From (\ref{Hsigma}), we have 
\begin{align*}
\inum\Lambda_{H_1}\partial \bar{\partial} \sigma_1\geq 3e^{\sigma_1}-4+M_2^{-1}.
\end{align*}
By the Cheng--Yau maximum principle, we have
\begin{align*}
0\geq 3M_1-4+M_2^{-1}.
\end{align*}
This implies
\begin{align*}
M_1\leq \frac{4}{3}-\frac{1}{3}M_2^{-1} \ \text{and} \ M_1B_1\leq B_2.
\end{align*}
Therefore we obtain (\ref{M}) and (\ref{MB}) for the case where $j=1$.
Suppose that (\ref{M}) holds for some $j=k$, where $1\leq k<n-1$. We show that this implies inequality (\ref{MB}) for $j=k+1$. From (\ref{Hsigma2}), we have 
\begin{align*}
\inum \Lambda_{H_{k+1}}\partial\bar{\partial}\sigma_{k+1} \geq (3-M_k)e^{\sigma_{k+1}}-3+M_{k+2}^{-1}.
\end{align*}
Since we have (\ref{M}) for the case where $j=k$, it holds that $3-M_k>0$. Then by the Cheng--Yau maximum principle, we have
\begin{align*}
0\geq (3-M_k)M_{k+1}-3+M_{k+2}^{-1}.
\end{align*}
This implies 
\begin{align*}
M_{k+1}B_{k+1} \leq B_{k+2}.
\end{align*}
Suppose that (\ref{M}) holds for $j=n-1$. We show that this implies inequality (\ref{MB}) for $j=n$. From (\ref{Hsigma3}), we have
\begin{align*}
\inum \Lambda_{H_n}\partial\bar{\partial} \sigma_n\geq (4-\delta-M_{n-1})e^{\sigma_n}-(3-\delta). 
\end{align*}
Since we have (\ref{M}) for the case where $j=n-1$, it holds that $4-\delta-M_{n-1}>0$. Then by the Cheng--Yau maximum principle, we have 
\begin{align*}
0\geq (4-\delta-M_{n-1})M_n-(3-\delta).
\end{align*}
This implies (\ref{MB}) for the case where $j=n$. Suppose that (\ref{MB}) holds for some $j=k$, where $2\leq k\leq n-1$ and that inequality (\ref{M}) holds for $j=k-1$. We show that this implies (\ref{M}) holds for $j=k$, which completes the proof of (\ref{M}) and (\ref{MB}). From (\ref{MB}) for $j=k$, we have
\begin{align}
M_k(3-M_{k-1})\leq 3-M_{k+1}^{-1}. \label{mk}
\end{align}
By combining (\ref{mk}) with (\ref{M}) for $j=k-1$, we have
\begin{align*}
M_k\left(3-\frac{4k-4}{2k-1}+\frac{2k-3}{2k-1}M_k^{-1}\right)\leq 3-M_{k+1}^{-1}.
\end{align*}
This implies
\begin{align*}
M_k \leq \frac{4k}{2k+1}-\frac{2k-1}{2k+1}M_{k+1}^{-1}.
\end{align*} 
By induction on $j$, we obtain (\ref{M}) and (\ref{MB}) from the above. We next prove (\ref{Mp}) and (\ref{MpBp}). From (\ref{Hsigma5}), we have 
\begin{align*}
\inum \Lambda_{H_{n-1}}\partial \bar{\partial} \sigma^\prime_{n-1}\geq (3-\delta)e^{\sigma_{n-1}^\prime}-(4-\delta)+M_{n-2}^{\prime-1}. 
\end{align*}
Then by the Cheng--Yau maximum principle, we have
\begin{align*}
0\geq (3-\delta)M_{n-1}^\prime-(4-\delta)+M_{n-2}^{\prime-1}.
\end{align*}
This implies
\begin{align*}
M_{n-1}^\prime\leq \frac{4-\delta}{3-\delta}-\frac{M_{n-2}^{\prime-1}}{3-\delta} \ \text{and} \
M_{n-1}^\prime B_{n-1}^\prime\leq B_{n-2}^\prime,
\end{align*}
and thus we have proved (\ref{Mp}) and (\ref{MpBp}) for the case where $j=1$. Suppose that (\ref{Mp}) holds for some $j=k$, where $1\leq k<n-2$. We show that this implies (\ref{MpBp}) for $j=k+1$. From (\ref{Hsigma4}), we have
\begin{align*}
\inum \Lambda_{H_{n-(k+1)}}\partial \bar{\partial} \sigma^\prime_{n-(k+1)}\geq -3+3e^{\sigma_{n-(k+1)}^\prime}+M_{n-(k+2)}^{\prime-1}-M_{n-k}^\prime e^{\sigma_{n-(k+1)}^\prime}
\end{align*}
Since we have (\ref{Mp}) for $j=k$, it holds that $3-M_{n-k}^\prime>0$. Then by the Cheng--Yau maximum principle, we have
\begin{align*}
0\geq -3+3M_{n-(k+1)}^\prime+M_{n-(k+2)}^{\prime-1}-M_{n-k}^\prime M_{n-(k+1)}^\prime.
\end{align*}
This implies $M_{n-(k+1)}^\prime B_{n-(k+1)}^\prime\leq B_{n-(k+2)}^\prime$. Suppose that (\ref{Mp}) holds for $j=n-2$. We show that this implies (\ref{MpBp}) for $j=n-1$. From (\ref{Hs}), we have 
\begin{align*}
\inum \Lambda_{H_1}\partial \bar{\partial} \sigma^\prime_1\geq -3+3e^{\sigma_1^\prime}-e^{\sigma_1^\prime}M_2^\prime.
\end{align*}
Since we have (\ref{Mp}) for $j=n-2$, it holds that $3-M_2^\prime>0$. Then by the Cheng-Yau maximum principle, we have
\begin{align*}
0\geq -3+3M_1^\prime -M_1^\prime M_2^\prime.
\end{align*}
This implies $M_1^\prime B_1^\prime\leq B_0^\prime$. Suppose that (\ref{MpBp}) holds for some $j=k$, where $2\leq k\leq n-2$, and (\ref{Mp}) holds for $j=k-1$. We show that this implies (\ref{Mp}) for $j=k$, which completes the proof of (\ref{Mp}) and (\ref{MpBp}). From (\ref{MpBp}) for $j=k$, we have
\begin{align*}
M_{n-k}^\prime(3-M_{n-(k-1)}^\prime)\leq {3-M_{n-(k+1)}^{\prime-1}}.
\end{align*}
Since we have (\ref{Mp}) for $j=k-1$, it holds that
\begin{align*}
M_{n-k}^\prime\left(3-\frac{4(k-1)-(2k-3)\delta}{2k-1-(k-1)\delta}+\frac{2k-3-(k-2)\delta}{2k-1-(k-1)\delta}M_{n-k}^{\prime-1}\right)\leq {3-M_{n-(k+1)}^{\prime-1}}.
\end{align*}
This implies
\begin{align*}
M_{n-k}^\prime \leq \frac{4k-(2k-1)\delta}{2k+1-k\delta}-\frac{2k-1-(k-1)\delta}{2k+1-k\delta}M_{n-(k+1)}^{\prime-1}.
\end{align*}
Therefore we obtain (\ref{Mp}) for $j=k$. This implies (\ref{Mp}) and (\ref{MpBp}). 
\end{proof}

\begin{lemm}
{\it Suppose that $e^\varphi$ is smooth. Then the following holds:
\begin{align}
&M_j-1\leq \frac{2j-1}{2} B_{j+1},\ \text{for $j=1,\dots, n$}\label{mj1}\\
&M_{n-j}^\prime-1\leq \frac{2j-1-(j-1)\delta}{2-\delta}B_{n-(j+1)}^\prime\ \text{for $j=1,\dots, n-1$}.\label{mjp1}
\end{align}
}
\end{lemm}
\begin{proof}
We provide proofs of (\ref{mj1}) and (\ref{mjp1}) inductively with respect to $j$, which are identical to those in \cite[Proof of Claim 4.2 (3)]{LM1}. We first consider the case where $j=1$. From the definition of $B_1$ and $B_{n-1}^\prime$, we have the following:
\begin{align*}
M_1-1&=M_1(1-M_1^{-1}) \\
&=\frac{1}{2}M_1B_1 \\
&\leq \frac{1}{2}B_2, \\
M_{n-1}^\prime-1&=M_{n-1}^\prime(1-M_{n-1}^{\prime-1}) \\
&=\frac{1}{2-\delta}M_{n-1}^\prime B_{n-1}^\prime \\
&\leq \frac{1}{2-\delta} B_{n-2}^\prime,
\end{align*}
and thus we have (\ref{mj1}) and (\ref{mjp1}) for the case where $j=1$.
We next suppose that (\ref{mj1}) holds for $j=k$, where $1\leq k\leq n-1$, and that (\ref{mjp1}) holds for $j=k$, where $1\leq k\leq n-2$. We show that (\ref{mj1}) holds for $j=k+1$. By the definition of $B_{k+1}$, we have 
\begin{align*}
M_{k+1}-1&=M_{k+1}(1-M_{k+1}^{-1}) \\
&=M_{k+1}(B_{k+1}-1+M_k) \\
&\leq M_{k+1}(B_{k+1}+\frac{2k-1}{2}B_{k+1}) \\
&=\frac{2k+1}{2}M_{k+1}B_{k+1} \\
&\leq \frac{2k+1}{2}B_{k+2}, \\
M_{n-(k+1)}^\prime -1&=M_{n-(k+1)}^\prime(1-M_{n-(k+1)}^{\prime-1}) \\
&=M_{n-(k+1)}^\prime(B_{n-(k+1)}^\prime-1+M_{n-k}^\prime) \\
&\leq M_{n-(k+1)}^\prime(B_{n-(k+1)}^\prime+\frac{2k-1-(k-1)\delta}{2-\delta}B_{n-(k+1)}^\prime) \\
&=\frac{2k+1-k\delta}{2-\delta}M_{n-(k+1)}^\prime B_{n-(k+1)}^\prime \\
&\leq \frac{2k+1-k\delta}{2-\delta}B_{n-(k+2)}^\prime.
\end{align*}
This implies (\ref{mj1}) and (\ref{mjp1}) for $j=k+1$, and thus we have completed the proofs of (\ref{mj1}) and (\ref{mjp1}). 
\end{proof}
\begin{lemm}\label{26}{\it Suppose that $\varphi$ is smooth. Then the following holds:
\begin{align}
&M_j\leq 1\ \text{for $j=1,\dots, n$}, \label{MJ1} \\
&M_j^\prime\leq \frac{\lambda_{j+1}}{\lambda_j} \ \text{for $j=1,\dots, n-1$}. \label{lambdajj}
\end{align}
}
\end{lemm}
\begin{proof} The proof is identical to that in \cite[pp.~23--24]{LM1}. We first prove (\ref{MJ1}). From (\ref{mj1}) for $j=n$, we have 
\begin{align*}
M_n-1\leq \frac{2n-1}{2}(2-\delta)(1-M_n).
\end{align*}
This implies $M_n\leq 1$. Suppose that (\ref{MJ1}) holds for some $j=k$, where $2\leq k\leq n$. We show that this implies (\ref{MJ1}) for $j=k-1$. From (\ref{mj1}) for $j=k-1$, we have 
\begin{align*}
M_{k-1}-1&\leq \frac{2k-3}{2}(2-M_k^{-1}-M_{k-1}) \\
&\leq \frac{2k-3}{2}(1-M_{k-1}).
\end{align*}
This implies $M_{k-1}\leq 1$, and thus we have proved (\ref{MJ1}). We next prove (\ref{lambdajj}). We introduce the following symbols:
\begin{align*}
d_j^\prime&\coloneqq \frac{\lambda_{j+1}}{\lambda_j}\ \text{for $j=1,\dots, n-1$}, \\
D_0^\prime&\coloneqq 2-d_1^\prime, \\
D_j^\prime&\coloneqq 2-d_j^{\prime-1}-d_{j+1}^\prime \ \text{for $j=1,\dots, n-2$}, \\
D_{n-1}^\prime&\coloneqq (2-\delta)(1-d_{n-1}^{\prime-1}).
\end{align*}
Then by the definition of $d_1^\prime,\dots, d_{n-1}^\prime$, the following holds:
\begin{align}
&0=-3+3d_1^\prime-d_1^\prime d_2^\prime, \\
&0=-3+3d_j^\prime+d_{j-1}^{\prime-1}-d_j^\prime d_{j+1}^\prime\ \text{$j=2,\dots, n-2$}, \\
&0=(3-\delta)d_{n-1}^\prime-(4-\delta)+d_{n-2}^{\prime-1}. 
\end{align}
Therefore by the definition of $D_1^\prime,\dots, D_{n-1}^\prime$ and the proof of Lemma \ref{MMBB}, we can check that the following holds:
\begin{align}
d_{n-j}^\prime-1=\frac{2j-1-(j-1)\delta}{2-\delta}D_{n-(j+1)}^\prime\ \text{for $j=1,\dots, n-1$}. \label{djp1}
\end{align}
We provide a proof of (\ref{lambdajj}). From (\ref{mjp1}) and (\ref{djp1}) for $j=n-1$, we have
\begin{align*}
M_1^\prime-1&\leq \frac{2n-3-(n-2)\delta}{2-\delta} (2-M_1^\prime), \\
d_1^\prime-1&=\frac{2n-3-(n-2)\delta}{2-\delta} (2-d_1^\prime).
\end{align*}
This implies $M_1^\prime\leq d_1^\prime$. Suppose that (\ref{lambdajj}) holds for some $j=k$, where $1\leq k\leq n-2$. We show that this implies (\ref{lambdajj}) for $j=k+1$. From (\ref{mjp1}) and (\ref{djp1}) for $j=n-(k+1)$, we have 
\begin{align*}
M_{k+1}^\prime-1&\leq \frac{2(n-k)-3-(n-k-2)\delta}{2-\delta}(2-M_k^{\prime-1}-M_{k+1}^\prime) \\
&\leq \frac{2(n-k)-3-(n-k-2)\delta}{2-\delta}(2-d_k^{\prime-1}-M_{k+1}^\prime), \\
d_{k+1}^\prime-1&=\frac{2(n-k)-3-(n-k-2)\delta}{2-\delta}(2-d_k^{\prime-1}-d_{k+1}^\prime).
\end{align*}
This implies $M_{k+1}^\prime\leq d_{k+1}^\prime$, which completes the proof of (\ref{lambdajj}).
\end{proof}
\begin{lemm}\label{lemma31}{\it Suppose that $\varphi$ is not identically $-\infty$ and that $e^{-\varphi}h_\refe$ is not induced by a flat complete K\"ahler metric. Suppose also that $\varphi$ is smooth. Then the following estimate holds:
\begin{align*}
&\frac{\lambda_{j-1}}{\lambda_j}=\frac{(j-1)(r-j+1)}{j(r-j)}<H_{j-1}\otimes H_j^{-1}<1 \ \text{for all $2\leq j\leq n$},\\
& H_r\otimes H_1^{-1}<1.
\end{align*}
}
\end{lemm}
\begin{proof} 
We first consider $\sigma_1,\dots, \sigma_n$. From (\ref{MJ1}), (\ref{Hsigma}), (\ref{Hsigma2}), and (\ref{Hsigma3}), we have
\begin{align}
&\inum\Lambda_{H_1}\partial \bar{\partial} \sigma_1\geq 3e^{\sigma_1}-3, \label{sigmad1}\\
&\inum \Lambda_{H_j}\partial\bar{\partial}\sigma_j \geq 2e^{\sigma_j}-2 \ \text{for $j=2,\dots, n-1$}. \label{sigmad2} \\
&\inum \Lambda_{H_n}\partial\bar{\partial} \sigma_n\geq (3-\delta)(e^{\sigma_n}-1). \label{sigmad3}
\end{align}
From (\ref{MJ1}), (\ref{sigmad1}), (\ref{sigmad2}), (\ref{sigmad3}), and the strong maximum principle, we see that for each $j=1,\dots, n$, either of the following holds: $e^{\sigma_j}<1$ or $e^{\sigma_j}=1$, where the latter equality means that $e^{\sigma_j}$ is identically 1 on $X$. If $e^{\sigma_1}=1$, then from (\ref{Hsigma}), we have $e^{\sigma_2}=1$. If $e^{\sigma_j}=1$ for some $j=2,\dots, n-1$, then from (\ref{Hsigma2}), we have $e^{\sigma_{j-1}}=e^{\sigma_{j+1}}=1$. If $e^{\sigma_n}=1$, then from (\ref{Hsigma3}), we have $e^{\sigma_{n-1}}=1$. Note that if $e^{\sigma_1}=\cdots =e^{\sigma_n}=1$, then $e^{-\varphi}h_\refe$ is flat. Therefore if $e^{-\varphi}h_\refe$ is non-flat, then we have $e^{\sigma_j}<1$ for all $j=1,\dots, n$. We next consider $\sigma_1^\prime,\dots, \sigma_{n-1}^\prime$. From (\ref{lambdajj}), (\ref{Hs}), (\ref{Hsigma4}), and (\ref{Hsigma5}), we have
\begin{align}
&\inum \Lambda_{H_1}\partial \bar{\partial} \sigma^\prime_1\geq -3+3e^{\sigma_1^\prime}-d_2^\prime e^{\sigma_1^\prime}, \label{sigmadp1} \\
&\inum \Lambda_{H_j}\partial \bar{\partial} \sigma^\prime_j\geq -3+3e^{\sigma_j^\prime}+d_{j-1}^{\prime-1}-d_{j+1}^\prime e^{\sigma_j^\prime}\ \text{$j=2,\dots, n-2$}, \label{sigmadp2}\\
&\inum \Lambda_{H_{n-1}}\partial \bar{\partial} \sigma^\prime_{n-1}\geq (3-\delta)e^{\sigma_{n-1}^\prime}-(4-\delta)+d_{n-2}^{\prime-1}. \label{sigmadp3}
\end{align}
From (\ref{lambdajj}), (\ref{sigmadp1}), (\ref{sigmadp2}), and (\ref{sigmadp3}), and the strong maximum principle, we see that for each $j=1,\dots, n-1$, either of the following holds: $e^{\sigma_j^\prime}<d_j^\prime$ or $e^{\sigma_j^\prime}=d_j^\prime$. If $e^{\sigma_1^\prime}=d_1^\prime$, then from (\ref{Hs}), we have $e^{\sigma_2^\prime}=d_2^\prime$. If $e^{\sigma_j^\prime}=d_j^\prime$ for some $j=2,\dots, n-2$, then from (\ref{Hsigma4}), we have $e^{\sigma_{j-1}^\prime}=d_{j-1}^\prime$ and $e^{\sigma_{j+1}^\prime}=d_{j+1}^\prime$. If $e^{\sigma_{n-1}^\prime}=d_{n-1}^\prime$, then from (\ref{Hsigma5}), we have $e^{\sigma_{n-2}^\prime}=d_{n-2}^\prime$. Note that if $e^{\sigma_j^\prime}=d_j^\prime$ for all $j=1,\dots, n-1$, then from (\ref{sigma4}) for $j=1$ we see that $\varphi$ is identically $-\infty$. Therefore if $\varphi$ is not identically $-\infty$, then we have $e^{\sigma_j^\prime}<d_j^\prime$ for all $j=1,\dots, n-1$. This implies the claim. 
\end{proof}

Then we prove Proposition \ref{r/2}.

\begin{proof}[Proof of Proposition \ref{r/2}] Let $h=(h_1,\dots, h_r)$ be a complete solution to equation (\ref{phi}) with $\det(h)=1$. Since it is sufficient to lift $X$ to the universal covering space and show the proposition, we will assume from the beginning that $X$ is either the complex plane $\C$ or the unit disc $\D$ and $h$ has an approximation $(h_\epsilon)_{\epsilon>0}$. We denote by $H_{j,\epsilon} \ (j=1,\dots, r-1)$ the complete metrics associated with $e^{-\varphi_\epsilon}h_\ast$ and by $H_{r,\epsilon}$ the metric $ H_{1,\epsilon}^{-1}\otimes \cdots \otimes H_{r-1,\epsilon}^{-1} \otimes (e^{-\varphi_\epsilon}h_\refe)^{-r}$. 
We have proved (\ref{Gauss}) and (\ref{Euler}) for the case where $\varphi$ is smooth. From (\ref{Gauss}) and (\ref{Euler}) for $e^{-\varphi_\epsilon}h_\refe$, we have
\begin{align}
&\frac{\lambda_{j-1}}{\lambda_j}=\frac{(j-1)(r-j+1)}{j(r-j)}<H_{j-1,\epsilon}\otimes H_{j,\epsilon}^{-1}<1 \ \text{for all $2\leq j\leq n$},\label{Gausse}\\
& H_{r,\epsilon}\otimes H_{1,\epsilon}^{-1}<1. \label{Eulere}
\end{align}
By taking the limit $\epsilon\searrow 0$ of (\ref{Gausse}) and (\ref{Eulere}) at each point of $X$, we have 
\begin{align}
&\frac{\lambda_{j-1}}{\lambda_j}=\frac{(j-1)(r-j+1)}{j(r-j)}\leq H_{j-1}\otimes H_j^{-1}\leq 1 \ \text{for all $2\leq j\leq n$},\label{Gauss0}\\
& H_r\otimes H_1^{-1}\leq 1. \label{Euler0}
\end{align}
Our goal is to show that $e^{\sigma_j} < 1$ for each $j = 1, \dots, n$, and that $e^{\sigma_j^\prime} < d_j^\prime = \frac{\lambda_{j+1}}{\lambda_j}$ for each $j = 1, \dots, n - 1$.  
Since the proofs are essentially the same, we only show that $e^{\sigma_1}<1$. We use the mean-value inequality (see Section \ref{CYMMVI}). We set $\sigma_{j,\epsilon}\coloneqq \log(H_{j-1,\epsilon}\otimes H_{j,\epsilon}^{-1})$ for $j=1,2$.
Since we have $e^{\sigma_{1,\epsilon}}, e^{\sigma_{2,\epsilon}}\leq 1$, from (\ref{Hsigma}), it holds that
\begin{align}
\inum\Lambda_{H_{1,\epsilon}}\partial \bar{\partial} (e^{\sigma_{1,\epsilon}}-1)&\geq e^{\sigma_{1,\epsilon}}(3e^{\sigma_{1,\epsilon}}-3)\geq 3(e^{\sigma_{1,\epsilon}}-1). \label{es-1}
\end{align}
We take a smooth K\"ahler metric $g$ on $X$ that is mutually bounded
with $H_1$. Since $H_1$ is complete, $g$ is also complete. Let $x_0\in X$. By the local convergence of
$H_{1,\epsilon}$ to $H_1$, there exists a constant $C_0>1$,
independent of sufficiently small $\epsilon$, such that
\begin{align*}
C_0^{-1}g\leq H_{1,\epsilon}\leq C_0g
\end{align*}
on $B_g(x_0,1)$. We set $u\coloneqq 1-e^{\sigma_1}$ and $u_\epsilon\coloneqq 1-e^{\sigma_{1,\epsilon}}$. 
By the uniform mutual boundedness of $g$ and $H_{1,\epsilon}$, (\ref{es-1}) implies that there exists a positive constant $c$, independent of $\epsilon$, such that
\begin{align*}
&\Delta_gu_\epsilon\leq cu_\epsilon. 
\end{align*}
Since $u_\epsilon$ converges to $u$ everywhere on $\overline{B}_g(x_0,1)$, by Lemma \ref{wMV}, there exists a positive constant $C$ and $0<p<1$ such that
\begin{align}
u(x_0)\geq C\left(\int_{B_g(x_0,1/2)}u^pd\mu_g\right)^{1/p}. \label{ux0}
\end{align}
Define
\begin{align*}
A_u\coloneqq \{x\in X\mid u(x)=0\}. 
\end{align*}
Since $u$ is nonnegative, $A_u$ coincides with the set $\{x\in X\mid u(x)\leq 0\}$. Also, since $\sigma_1$ is upper semicontinuous, $u$ is lower semicontinuous, and hence, $A_u$ is a closed set. If $x_0\in A_u$, then (\ref{ux0}) implies that $u=0$ almost everywhere
on $B_g(x_0,1/2)$. Since $u$ is lower semicontinuous and
nonnegative, it follows that $u=0$ everywhere on
$B_g(x_0,1/2)$. Hence $A_u$ is open. By the same reason as in the proof of Lemma \ref{lemma31}, $u$ is not identically 0. Therefore, it follows that $A_u$ is an empty set, and thus, we have $e^{\sigma_1}<1$. 
\end{proof}

\begin{lemm}\label{PQ}{\it Let $P_0\leq P_1\leq \cdots \leq P_{r-1}$ and $Q_0\leq Q_1\leq \dots \leq Q_{r-1}$ be non-negative real numbers such that $P_0+\cdots +P_{r-1}=1$, $Q_0+\cdots +Q_{r-1}=1$. Suppose that $P_1,\dots, P_{r-1}$ and $Q_1,\dots, Q_{r-1}$ are strictly positive and that the following holds:
\begin{align*}
Q_j/Q_{j+1}\leq P_j/P_{j+1} \ \text{for all $j=0,\dots, r-2$.}
\end{align*}
Let $S$ be a strictly Schur-concave function. Then the function $S$ satisfies 
\begin{align*}
S(Q_0, Q_1,\dots, Q_{r-1})\leq S(P_0, P_1,\dots, P_{r-1}).
\end{align*}
Moreover, the equality is achieved if and only if $Q_j/Q_{j+1}=P_j/P_{j+1} $ holds for all $j=0,\dots, r-2$.
} 
\end{lemm}
\begin{proof} By induction on $j-i$, the assumption on the adjacent ratios implies
\begin{align}
\frac{Q_i}{Q_j}\leq\frac{P_i}{P_j} \label{PiPj}
\end{align}
for all $0\leq i<j\leq r-1$.
Indeed, the case $j-i=1$ is clear from the assumption, and we suppose that inequality (\ref{PiPj}) holds for the case $j-i=k$. Then if $j-i=k+1$, we obtain
\begin{align*}
\frac{Q_i}{Q_j}&=\frac{Q_i}{Q_{j-1}}\cdot \frac{Q_{j-1}}{Q_j} \leq \frac{P_i}{P_{j-1}}\cdot \frac{P_{j-1}}{P_j} =\frac{P_i}{P_j}. 
\end{align*}
Therefore by induction, we have (\ref{PiPj}) for all $0\leq i<j\leq r-1$. Hence, for every $k=0,\ldots,r-2$ and all $0\leq i\leq k<j\leq r-1$, we have
\begin{align*}
Q_iP_j\leq P_iQ_j. 
\end{align*}
Summing these inequalities over $i$ and $j$, we obtain
\begin{align*}
\left(\sum_{i=0}^kQ_i\right) \left(\sum_{j=k+1}^{r-1}P_j\right) \leq \left(\sum_{i=0}^kP_i\right) \left(\sum_{j=k+1}^{r-1}Q_j\right). 
\end{align*}
Since both $P_0,\dots, P_{r-1}$ and $Q_0, Q_1,\dots, Q_{r-1}$ have total sum 1, this is equivalent to
\begin{align*}
\sum_{i=0}^kQ_i\leq\sum_{i=0}^kP_i. 
\end{align*}
Since the components of $P$ and $Q$ are arranged in non-decreasing order, it follows that $Q\succ P$. If $S$ is a Schur-concave function, then we obtain
\begin{align*}
S(Q_0, Q_1,\dots, Q_{r-1})\leq S(P_0, P_1,\dots, P_{r-1}).
\end{align*}
Suppose that $S$ is a strictly Schur-concave function and suppose that $S(Q_0, Q_1,\dots, Q_{r-1})=S(P_0, P_1,\dots, P_{r-1})$ holds. Then by definition of strict Schur-concavity, $Q_0,\dots, Q_{r-1}$ is a permutation of $P_0,\dots, P_{r-1}$. Since both $P_0,\dots, P_{r-1}$ and $Q_0,\dots, Q_{r-1}$ are arranged in non-decreasing order, this is equivalent to $Q_j=P_j$ for all $j=0,\dots, r-1$. Moreover, this clearly implies
\begin{align*}
\frac{Q_j}{Q_{j+1}}=\frac{P_j}{P_{j+1}} 
\end{align*}
for every $j=0,\ldots,r-2$. Conversely, if these equalities hold for every $j$, then $Q_j/P_j$ is independent of $j$. Since
\begin{align*}
\sum_{j=0}^{r-1}P_j=\sum_{j=0}^{r-1}Q_j=1, 
\end{align*}
we obtain $Q_j=P_j$ for all $j=0,\dots, r-1$. This proves the assertion on the equality case. 
\end{proof}
From Proposition \ref{r/2} and Lemma \ref{PQ}, we have Theorem \ref{main theorem 1}.

\subsection{Proof of Theorem \ref{main theorem 2}}
We next prove Theorem \ref{main theorem 2}.
\begin{proof}[Proof of Theorem \ref{main theorem 2}] 
We set $Z_{r,\beta}$ and $h_{r,\beta}$ as follows:
\begin{align*}
&Z_{r,\beta}\coloneqq \sum_{j=1}^{r-1}\lambda_j^\beta, \\
&h_{r,\beta}\coloneqq \sum_{j=1}^{r-1}\frac{1}{r}\left(\frac{\lambda_j^\beta}{r^{2\beta}}\right) \log\left(\frac{\lambda_j^\beta}{r^{2\beta}}\right)=\beta\frac{\partial}{\partial\beta}(Z_{r,\beta}/r^{2\beta+1}).
\end{align*}
Then $S_{r,\beta}$ is calculated as follows:
\begin{align}
S_{r,\beta}&=-\sum_{j=1}^{r-1}\left(\frac{\lambda_j^\beta}{Z_{r,\beta}}\right)\log\left(\frac{\lambda_j^\beta}{Z_{r,\beta}}\right) \notag \\ 
&=-\sum_{j=1}^{r-1}\frac{\lambda_j^\beta}{Z_{r,\beta}} \left(\log\frac{\lambda_j^\beta}{r^{2\beta}}-\log r+\log\frac{r^{2\beta+1}}{Z_{r,\beta}}\right) \notag \\
&=-\frac{r^{2\beta+1}}{Z_{r,\beta}}\sum_{j=1}^{r-1}\frac{1}{r}\left(\frac{\lambda_j^\beta}{r^{2\beta}}\right) \log\left(\frac{\lambda_j^\beta}{r^{2\beta}}\right)+\log r-\log\left(\frac{r^{2\beta+1}}{Z_{r,\beta}}\right) \notag \\
&=-\frac{h_{r,\beta}}{(Z_{r,\beta}/r^{2\beta+1})}+\log r+\log\left(\frac{Z_{r,\beta}}{r^{2\beta+1}}\right). \label{Srb}
\end{align}
We define constants $c_\beta$ and $d_\beta$, which may not be finite, as follows:
\begin{align*}
&c_\beta\coloneqq \int_0^1s^\beta(1-s)^\beta ds, \\
&d_\beta\coloneqq \int_0^1s^\beta(1-s)^\beta \log s\ ds.
\end{align*}
Note that each of the above constants is finite if and only if $\beta>-1$, and if $\beta\leq -1$, then $c_\beta=\infty$ and $d_\beta=-\infty$. Suppose that $\beta$ is positive. Then since $Z_{r,\beta}/r^{2\beta+1}$ and $h_{r,\beta}$ are Riemann sums of $f_\beta(s)\coloneqq s^\beta(1-s)^\beta$ and $g_\beta (s)\coloneqq s^\beta(1-s)^\beta\log s^\beta(1-s)^\beta$ on the interval $0\leq s\leq 1$, respectively, we have 
\begin{align*}
\lim_{r\to\infty} (S_{r,\beta}-\log r)=-\frac{2\beta d_\beta}{c_\beta}+\log(c_\beta).
\end{align*}
We next consider the case where $\beta$ is negative. The following estimate holds:
\begin{align}
Z_{r,\beta} /r^{2\beta+1}
&\geq \int_{1/r}^{n/r}s^\beta(1-s)^\beta ds+\int_{n/r}^{1-1/r}s^\beta(1-s)^\beta ds \notag \\
&=\int_{1/r}^{1-1/r}s^\beta(1-s)^\beta ds, \label{1-s}\\
Z_{r,\beta}/r^{2\beta+1}&\leq \int_{1/r}^{n/r}s^\beta (1-s)^\beta ds +\frac{1}{r^{2\beta+1}}(r-1)^\beta+\int_{n/r}^{1-1/r}s^\beta (1-s)^\beta ds+\frac{1}{r^{2\beta+1}}(r-1)^\beta \notag\\
&=\int_{1/r}^{1-1/r}s^\beta (1-s)^\beta ds+\frac{2}{r^{2\beta+1}}(r-1)^\beta \label{int}. 
\end{align}
Also, by the monotonicity of the function $g_\beta$ at both ends of the interval, there exists $\epsilon_\beta>0$ such that on the interval $0<s\leq \epsilon_\beta$ (resp. $1-\epsilon_\beta\leq s<1$), $g_\beta$ monotonically decreases (resp. monotonically increases). Therefore, if we take $r$ large enough and $j$ (resp. $k$) small enough (resp. close enough to $r$) so that $j/r\leq \epsilon_\beta$ (resp. $1-\epsilon_\beta\leq k/r$) holds, the following holds:
\begin{align}
&\int_{(j-1)/r}^{j/r}g_\beta(s)ds\geq \frac{1}{r}g_\beta(j/r)\geq \int_{j/r}^{(j+1)/r}g_\beta(s)ds, \label{jr}\\ 
&\int_{(k-1)/r}^{k/r}g_\beta(s)ds\leq \frac{1}{r}g_\beta(k/r)\leq \int_{k/r}^{(k+1)/r}g_\beta(s)ds. \label{kr}
\end{align}
We take $r$ large enough and set $j_0\coloneqq [r\epsilon_\beta ]+1$ and $k_0\coloneqq [r-r\epsilon_\beta]-1$, where we denote by $[\cdot]$ the Gaussian symbol. Note that we have $(j_0-1)/r\leq \epsilon_\beta$ and $(k_0+1)/r\leq 1-\epsilon_\beta$. Then we have the following estimate:
\begin{align}
&h_{r,\beta}\leq \frac{1}{r}g_\beta(1/r)+\int_{1/r}^{(j_0-1)/r}g_\beta(s)ds+\sum_{j=j_0}^{k_0}\frac{1}{r}g_\beta(j/r)+\int_{(k_0+1)/r}^{1-1/r}g_\beta(s)ds+\frac{1}{r}g_\beta(1-1/r), \label{grb}\\
&h_{r,\beta}\geq \int_{1/r}^{j_0/r}g_\beta(s)ds+\sum_{j=j_0}^{k_0}\frac{1}{r}g_\beta(j/r)+\int_{k_0/r}^{1-1/r}g_\beta(s)ds. \label{gbs}
\end{align}
From (\ref{grb}) and (\ref{gbs}), for the case where $\beta>-1$, we have
\begin{align}
\lim_{r\to\infty} h_{r,\beta}&=\int_0^{\epsilon_\beta} g_\beta(s)ds+\int_{\epsilon_\beta}^{1-\epsilon_\beta}g_\beta(s)ds +\int_{1-\epsilon_\beta}^1g_\beta(s)ds, \notag \\
&=\int_0^1g_\beta(s)ds \noindent \\
&=2\beta\int_0^1s^\beta(1-s)^\beta\log sds, \label{limr}
\end{align}
where we have used the easily verifiable fact that for the case where $\beta>-1$, the terms $\frac{1}{r}g_\beta(1/r)=\frac{1}{r}g_\beta(1-1/r)$ go to $0$ as $r\to\infty$.
Also, from (\ref{gbs}), for the case where $\beta\leq -1$, we have the following for some positive constants $C_{-1}$ and $C_\beta$:
\begin{align}
h_{r,\beta}\geq \begin{cases}
C_{-1}(\log r)^2+o((\log r)^2) \ &\text{if $\beta=-1$}\\
C_\beta r^{-\beta-1}\log r+o(r^{-\beta-1}\log r) \ &\text{if $\beta< -1$}. 
\end{cases} \label{hrbeta}
\end{align}
Then we prove the assertion for the case where $\beta>-1$. From (\ref{Srb}), (\ref{1-s}), (\ref{int}), and (\ref{limr}), we have 
\begin{align*}
\lim_{r\to\infty} (S_{r,\beta}-\log r)=-\frac{2\beta d_\beta}{c_\beta}+\log(c_\beta).
\end{align*}
We next suppose that $\beta=-1$. From (\ref{1-s}) and (\ref{int}), there exists a positive constant $C_{-1}^\prime$ such that the following holds for sufficiently large $r$:
\begin{align}
Z_{r,\beta} /r^{2\beta+1}\leq C_{-1}^\prime \log r.\label{Cprime}
\end{align}
Combining (\ref{Cprime}) with (\ref{Srb}) and (\ref{hrbeta}), we have
\begin{align*}
S_{r,\beta}-\log r\leq -\frac{C_{-1}(\log r)^2+o((\log r)^2)}{C_{-1}^\prime\log r}+\log\left(C_{-1}^\prime\log r\right) 
\end{align*}
This implies $\lim_{r\to\infty}(S_{r,\beta}-\log r)=-\infty$. Finally, we consider the case $\beta<-1$. We set $W_{r,\beta}\coloneqq Z_{r,\beta}/r^\beta$. By definition of $h_{r,\beta}$, we obtain 
\begin{align*}
h_{r,\beta}&=\beta\frac{\partial}{\partial\beta}(Z_{r,\beta}/r^{2\beta+1}) \\
&=\beta(W^\prime_{r,\beta}/r^{\beta+1}-(W_{r,\beta}/r^{\beta+1})\log r).
\end{align*}
Therefore, by (\ref{Srb}), we obtain
\begin{align*}
S_{r,\beta}&=-\frac{r^{\beta+1}h_{r,\beta}}{W_{r,\beta}}+\log r+\log W_{r,\beta}-(\beta+1)\log r \\
&=-\beta(W^\prime_{r,\beta}/W_{r,\beta})+\log W_{r,\beta}.
\end{align*}
We prove that as $r\to\infty$, $W_{r,\beta}$ converges to $2\zeta(-\beta)=2\sum_{j=1}^\infty j^\beta$ and $W^\prime_{r,\beta}$ converges to $-2\zeta^\prime(-\beta)=2\sum_{j=1}^\infty j^\beta\log j$, where $\zeta(-\beta)$ is the Riemann zeta function. Indeed, by using symmetry, each of $W_{r,\beta}$ and $W^\prime_{r,\beta}$ can be decomposed as follows:
\begin{align*}
W_{r,\beta}&=2\sum_{j=1}^{m_r}j^\beta(1-j/r)^\beta+O(r^\beta), \\
W^\prime_{r,\beta}&=2\sum_{j=1}^{m_r}j^\beta(1-j/r)^\beta(\log j+\log(1-j/r))+O(r^\beta \log r),
\end{align*} 
where $m_r$ is defined to be $m_r\coloneqq [r/2]$. Set  
\begin{align*}
&F_r(j)\coloneqq
\begin{cases}
j^\beta(1-j/r)^\beta\ \text{for}\ j\leq m_r \\
0 \ \text{for}\ j>m_r,
\end{cases} \\
&G_r(j)\coloneqq
\begin{cases}
j^\beta(1-j/r)^\beta(\log j+\log(1-j/r))\ \text{for}\ j\leq m_r \\
0 \ \text{for}\ j>m_r.
\end{cases}
\end{align*}
Then for a fixed $j$, as $r\to\infty$, $F_r(j)$ and $G_r(j)$ converge to $j^\beta$ and $j^\beta\log j$, respectively. Also, observe that 
\begin{align*}
(1-j/r)^\beta&\leq 2^{-\beta}, \\
-\log(1-j/r)&\leq \log 2
\end{align*}
for each $j=1,\dots, m_r$. Therefore we obtain $|F_r(j)|\leq 2^{-\beta}j^\beta$, $|G_r(j)|\leq 2^{-\beta}j^\beta\log j+2^{-\beta}j^\beta\log 2$.
Since $\beta<-1$, both $\sum_{j=1}^{\infty}j^\beta$ and $\sum_{j=1}^{\infty}j^\beta\log j$ converge. By the Lebesgue dominated convergence theorem, we obtain 
\begin{align*}
&\lim_{r\to\infty}W_{r,\beta}=2\zeta(-\beta), \\
&\lim_{r\to\infty}W_{r,\beta}^\prime=-2\zeta^\prime(-\beta),
\end{align*}
and thus we have
\begin{align*}
\lim_{r\to\infty}S_{r,\beta}=\beta(\zeta^\prime(-\beta)/\zeta(-\beta))+\log(2\zeta(-\beta)).
\end{align*}
In particular, we have $\lim_{r\to\infty}(S_{r,\beta}-\log r)=-\infty$. This implies the claim.
\end{proof}

\noindent
E-mail address 1: natsuo.miyatake.e8@tohoku.ac.jp

\noindent
E-mail address 2: natsuo.m.math@gmail.com \\

\noindent
Mathematical Science Center for Co-creative Society, Tohoku University, 468-1 Aramaki Azaaoba, Aoba-ku, Sendai 980-0845, Japan.


\begin{thebibliography}{99}
\bibitem{AF1}E. Aldrovandi and G. Falqui, {\it Geometry of Higgs and Toda fields on Riemann surfaces}, J. Geom. Phys. 17 (1995), no. 1, 25--48.
\bibitem{Bar1} D. Baraglia, {\it $G_2$ geometry and integrable systems}, arXiv:1002.1767 (2010).
\bibitem{BCHM1} T. Bayraktar, D. Coman, H. Herrmann, and G. Marinescu, {\it A survey on zeros of random holomorphic sections}, Dolomites Research Notes on Approximation 11(4) (2018), 1--19.
\bibitem{BH1} Y. Benoist and D. Hulin, {\it Cubic differentials and hyperbolic convex sets}, Journal of Differential Geometry 98.1 (2014): 1--19.
\bibitem{Ber1} R. J. Berman, {\it Determinantal point processes and fermions on complex manifolds: large deviations and bosonization}, Communications in Mathematical Physics 327.1 (2014): 1--47.
\bibitem{Ber2}R. J. Berman, {\it Bergman kernels and equilibrium measures for line bundles over projective manifolds}, American journal of mathematics 131.5 (2009): 1485--1524.
\bibitem{BB1} R. J. Berman and S. Boucksom, {\it Growth of balls of holomorphic sections and energy at equilibrium}, Inventiones mathematicae, 181(2), 337--394 (2010).
\bibitem{BBN1} R. J. Berman, S. Boucksom, and D. W. Nystr\"om, {\it Fekete points and convergence towards equilibrium measures on complex manifolds}, Acta Math. 207, 1--27 (2011).
\bibitem{BMPR1} B. Bianchini, L. Mari, P. Pucci, and M. Rigoli, {\it On the interplay among maximum principles, compact support principles and Keller--Osserman conditions on manifolds}, arXiv:1801.02102 (2018).
\bibitem{BL1} T. Bloom and N. Levenberg, {\it Random polynomials and pluripotential-theoretic extremal functions}, Potential Analysis 42.2 (2015): 311--334.
\bibitem{BL2}T. Bloom and N. Levenberg, {\it Weighted pluripotential theory in $\C^N$}, American journal of mathematics 125.1 (2003): 57--103.
\bibitem{BL3}T. Bloom and N. Levenberg, {\it Capacity convergence results and applications to a Berstein--Markov inequality}, Transactions of the American Mathematical Society 351.12 (1999): 4753--4767.
\bibitem{BS1} T. Bridgeland and I. Smith, {\it Quadratic differentials as stability conditions}, Publications math\'ematiques de l'IH\'ES 121 (2015): 155--278.
\bibitem{CY1} S. Y. Cheng and S. T. Yau, {\it Differential equations on Riemannian manifolds and their geometric applications}, Communications on Pure and Applied Mathematics 28.3 (1975): 333--354.
\bibitem{CT1} H. I. Choi and A. Treibergs, {\it Gauss maps of spacelike constant mean curvature hypersurfaces of Minkowski space}, Journal of differential geometry 32.3 (1990): 775--817.
\bibitem{Col1} B. Collier, {\it Finite order automorphisms of Higgs bundles: theory and application}, (Doctoral dissertation, University of Illinois at Urbana-Champaign).
\bibitem{CL1} B. Collier and Q. Li, {\it Asymptotics of Higgs bundles in the Hitchin component}, Advances in Mathematics 307 (2017): 488--558.
\bibitem{DL1} S. Dai and Q. Li, {\it Minimal surfaces for Hitchin representations}, Journal of Differential Geometry 112.1 (2019): 47--77.
\bibitem{DL2} S. Dai and Q. Li, {\it On cyclic Higgs bundles}, Math. Ann. 376 (2020), no. 3--4, 1225--1260.
\bibitem{DL3} S. Dai and Q. Li, {\it Bounded differentials on the unit disk and the associated geometry}, Transactions of the American Mathematical Society (2024).
\bibitem{Don1} S. K. Donaldson, {\it Boundary value problems for Yang--Mills fields}, Journal of geometry and physics 8.1--4 (1992): 89--122.
\bibitem{DW1} D. Dumas and M. Wolf, {\it Polynomial cubic differentials and convex polygons in the projective plane}, Geometric and Functional Analysis 25.6 (2015): 1734--1798.
\bibitem{FKMR1} O. El-Fallah, K. Kellay, J. Mashreghi, and T. Ransford, {\it A primer on the Dirichlet space} (Vol. 203). Cambridge University Press (2014).
\bibitem{GZ1} V. Guedj and A. Zeriahi, {\it Degenerate Complex Monge--Amp\`ere Equations}, EMS Tracts Math. 26, European Mathematical Society (EMS), Z\"urich 2017.
\bibitem{GH1} M. A. Guest and N. K. Ho, {\it Kostant, Steinberg, and the Stokes matrices of the $tt^\ast$-Toda equations}, Selecta Mathematica 25.3 (2019): 50.
\bibitem{GL1} M. A. Guest and C.-S. Lin, {\it Nonlinear PDE aspects of the $tt^\ast$ equations of Cecotti and Vafa}, J. Reine Angew. Math. 689 (2014), 1--32.

\bibitem{Hit1} N. J. Hitchin, {\it The self-duality equations on a Riemann surface}, Proceedings of the London Mathematical Society (3) 55(1) (1987):: 59--126.
\bibitem{Hit2} N. J. Hitchin, {\it Lie groups and Teichm\"uller space}, Topology, 31(3), 449--473 (1992).
\bibitem{JT1} A. Jaffe and C. Taubes, {\it Vortices and monopoles}, vol. 2 of Progress in Physics (1980).
\bibitem{Kus1} Y. Kusunoki, {\it Integrable holomorphic quadratic differentials with simple zeros}, Holomorphic Functions and Moduli I: Proceedings of a Workshop held March 13--19, 1986. Springer US, 1988.
\bibitem{LT1} F. Labourie and J. Toulisse, {\it Quasicircles and quasiperiodic surfaces in pseudo-hyperbolic spaces}, Inventiones mathematicae 233.1 (2023): 81--168.
\bibitem{LTW1} F. Labourie, J. Toulisse, and M. Wolf, {\it Plateau problems for maximal surfaces in pseudo-hyperbolic spaces}, Annales scientifiques de l'\'Ecole normale sup\'erieure (4) 57(2) (2024), 473--552.
\bibitem{LL1} L. D. Landau and E. M. Lifshitz, {\it Statistical Physics: Volume 5}, Vol. 5. Elsevier, 2013.
\bibitem{Lei1} T. Leinster, {\it Entropy and Diversity: The Axiomatic Approach}, Cambridge University Press (2020).
\bibitem{Li0} Q. Li, {\it Harmonic maps for Hitchin representations}, Geometric and Functional Analysis 29 (2019): 539--560.
\bibitem{Li1} Q. Li, {\it On the uniqueness of vortex equations and its geometric applications}, The Journal of Geometric Analysis 29 (2019): 105--120.
\bibitem{LM1} Q. Li and T. Mochizuki, {\it Complete Solutions of Toda Equations and Cyclic Higgs Bundles Over Non-compact Surfaces}, International Mathematics Research Notices 2025.7 (2025): rnaf081.
\bibitem{LM2} Q. Li and T. Mochizuki, {\it Isolated singularities of Toda equations and cyclic Higgs bundles}, arXiv:2010.06129 ver 3 (2020).
\bibitem{LM3} Q. Li and T. Mochizuki, {\it Higgs bundles in the Hitchin section over non‐compact hyperbolic surfaces}, Proceedings of the London Mathematical Society 129.6 (2024): e70008.
\bibitem{LNV1} G. Livan,  M. Novaes, and P. Vivo, {\it Introduction to Random Matrices: Theory and Practice}, SpringerBriefs in Mathematical Physics, vol. 26, Springer, 2018.
\bibitem{Lof1} J. C. Loftin, {\it Flat metrics, cubic differentials and limits of projective holonomies}, Geometriae Dedicata, 128 (2007), 97--106.
\bibitem{LoTW1} J. Loftin, A. Tamburelli, and M. Wolf, {\it Limits of cubic differentials and buildings}, Proceedings of the London Mathematical Society 132(4) (2026), e70154.
\bibitem{Ma1} W. Ma, {\it Harmonic metrics of $\mathrm {SO} _ {0}(n, n) $-Higgs bundles in the Hitchin section on non-compact hyperbolic surfaces}, Journal of Geometry and Physics 228 (2026), 105893.
\bibitem{MOA1} A. W. Marshall, I. Olkin, and B. C. Arnold, {\it Inequalities: theory of majorization and its applications}, 2nd ed., Springer Series in Statistics,
Springer, New York, 2011.
\bibitem{Mas1} H. Masur, {\it The growth rate of trajectories of a quadratic differential}, Ergodic Theory and Dynamical Systems 10.1 (1990): 151--176.
\bibitem{Miy2} N. Miyatake, {\it Generalizations of Hermitian--Einstein equation of cyclic Higgs bundles, their heat equation, and inequality estimates}, arXiv:2301.01584 (2023).
\bibitem{Miy3} N. Miyatake, {\it Cyclic Higgs bundles, subharmonic functions, and the Dirichlet problem}, Annals of Global Analysis and Geometry 67.1 (2025): 1--16.
\bibitem{Miy4} N. Miyatake, {\it Shannon entropy for harmonic metrics on cyclic Higgs bundles II}, arXiv:2508.12844.
\bibitem{Miy5} N. Miyatake, {\it Complete harmonic metrics and subharmonic functions on the unit disc}, arXiv:2508.12848. 
\bibitem{Miy6} N. Miyatake, {\it Complete harmonic metrics and subharmonic functions on the complex plane}, preprint. 
\bibitem{Miy7} N. Miyatake, {\it Equilibrium metrics and complete solutions}, in preparation.
\bibitem{Miy8} N. Miyatake, {\it Asymptotic formulas 
for 
the lower bound of the entropy functions}, in preparation. 
\bibitem{Moc0} T. Mochizuki, {\it Harmonic bundles and Toda lattices with opposite sign I}, arXiv:1301.1718.
\bibitem{Moc1} T. Mochizuki, {\it Harmonic bundles and Toda lattices with opposite sign II}, Communications in Mathematical Physics 328 (2014): 1159--1198.
\bibitem{Nie1} X. Nie, {\it Poles of cubic differentials and ends of convex $\mathbb {RP}^ 2$-surfaces}, Journal of Differential Geometry 123.1 (2023): 67--140.
\bibitem{Niel1}M. A. Nielsen, {\it An introduction to majorization and its applications to quantum mechanics},  Lecture Notes, Department of Physics, University of Queensland, Australia 53 (2002).
\bibitem{Ran1} T. Ransford, {\it Potential theory in the complex plane}, No. 28. Cambridge university press, 1995.
\bibitem{ST1} E. B. Saff and V. Totik, {\it Logarithmic potentials with external fields}, Vol. 316. Springer Science \& Business Media, 2013.
\bibitem{Sak1} K. Sakai, {\it Uniform degeneration of hyperbolic surfaces with boundary along harmonic map rays}, Journal of Topology and Analysis 18(4) (2026), 1097--1126.
\bibitem{Ser1} S. Serfaty, {\it Coulomb gases and Ginzburg--Landau vortices}, Zurich Lectures in Advanced Mathematics, European Mathematical Society, 2015.
\bibitem{Ser2} S. Serfaty, {\it Systems with Coulomb interactions}, Journ\'ees \'equations aux d\'eriv\'ees partielles (2014): 1--23.
\bibitem{Sh1} C. E. Shannon, {\it A mathematical theory of communication}, The Bell system technical journal, 27(3), 379--423 (1948).
\bibitem{SZ1} B. Shiffman and S. Zelditch, {\it Stochastic K\" ahler geometry: from random zeros to random metrics}, Surveys in Differential Geometry 26(1) (2021), 299--334.
\bibitem{Sim3} C. T. Simpson, {\it Katz's middle convolution algorithm}, Pure and Applied Mathematics Quarterly 5(2) (2009), 781--852.
\bibitem{Wan1} T. Y.-H. Wan, {\it Constant mean curvature surface, harmonic maps, and universal Teichm\"uller space}, J. Differential Geom.35 (1992), no.3, 643--657.
\bibitem{WA1} T. Y.-H. Wan and T. K.-K. Au, {\it Parabolic constant mean curvature spacelike surfaces}, Proceedings of the American Mathematical Society 120.2 (1994): 559--564.
\bibitem{Wol1} M. Wolf, {\it The Teichm\"uller theory of harmonic maps}, Journal of differential geometry 29.2 (1989): 449--479.
\end{thebibliography}
\end{document}